\documentclass[a4paper]{amsart}
\usepackage{amsmath,amsthm,amssymb,amsfonts,mathtools}
\usepackage{enumerate}
\usepackage[margin=3.75cm]{geometry}

\newtheorem{theorem}{Theorem}[section]
\newtheorem{lemma}[theorem]{Lemma}
\newtheorem*{lemma*}{Lemma}
\newtheorem{proposition}[theorem]{Proposition}
\newtheorem{corollary}[theorem]{Corollary}

\theoremstyle{definition}
\newtheorem{remark}[theorem]{Remark}

\newcommand{\C}{\mathbb{C}}
\newcommand{\R}{\mathbb{R}}
\newcommand{\N}{\mathbb{N}}
\newcommand{\Z}{\mathbb{Z}}

\newcommand{\hwv}{\gamma}				  
\newcommand{\Sphere}{\mathbb{S}}          
\newcommand{\weight}{\omega}              

\DeclareMathOperator{\tr}{tr}             
\newcommand{\id}{\mathrm{id}}    		  
\newcommand{\HS}{\mathrm{HS}}             
\DeclareMathOperator*{\esssup}{ess\,sup}  
\DeclareMathOperator{\supp}{supp}         
\newcommand{\Hom}{\mathrm{Hom}}           
\newcommand{\Lin}{\mathcal{L}}            

\newcommand{\FE}{\mathcal{E}}             
\newcommand{\FF}{\mathcal{F}}             
\newcommand{\AbsOp}{\Delta}               
\newcommand{\CRL}{\mathbb{L}}             
\newcommand{\dummy}{\,\cdot\,}
\newcommand{\group}[1]{\mathrm{#1}}       

\newcommand{\dist}{\varrho}   			  
\newcommand{\card}{\operatorname{card}}            
\renewcommand{\Re}{\operatorname{Re}}              
\renewcommand{\Im}{\operatorname{Im}}              
\newcommand{\dbarb}{\Bar{\partial}_b}         
\newcommand{\dbarba}{\Bar{\partial}_b^+}      
\newcommand{\boxb}{\Box_b}                         

\newcommand{\Forms}[1]{\mathrm{B}_{#1}}           

\newcommand{\TT}{Y}                      

\makeatletter
\DeclareFontFamily{OMX}{MnSymbolE}{}
\DeclareSymbolFont{MnLargeSymbols}{OMX}{MnSymbolE}{m}{n}
\SetSymbolFont{MnLargeSymbols}{bold}{OMX}{MnSymbolE}{b}{n}
\DeclareFontShape{OMX}{MnSymbolE}{m}{n}{
	<-6>  MnSymbolE5
	<6-7>  MnSymbolE6
	<7-8>  MnSymbolE7
	<8-9>  MnSymbolE8
	<9-10> MnSymbolE9
	<10-12> MnSymbolE10
	<12->   MnSymbolE12
}{}
\DeclareFontShape{OMX}{MnSymbolE}{b}{n}{
	<-6>  MnSymbolE-Bold5
	<6-7>  MnSymbolE-Bold6
	<7-8>  MnSymbolE-Bold7
	<8-9>  MnSymbolE-Bold8
	<9-10> MnSymbolE-Bold9
	<10-12> MnSymbolE-Bold10
	<12->   MnSymbolE-Bold12
}{}

\let\llangle\@undefined
\let\rrangle\@undefined
\DeclareMathDelimiter{\llangle}{\mathopen}%
{MnLargeSymbols}{'164}{MnLargeSymbols}{'164}
\DeclareMathDelimiter{\rrangle}{\mathclose}%
{MnLargeSymbols}{'171}{MnLargeSymbols}{'171}
\makeatother

\newcommand{\lnorm}{\left\Vert}
\newcommand{\rnorm}{\right\Vert}
\newcommand{\biglnorm}{\bigl\Vert}
\newcommand{\bigrnorm}{\bigr\Vert}
\newcommand{\lset}{\left\{}
\newcommand{\rset}{\right\}}
\newcommand{\Biglset}{\Bigl\{}
\newcommand{\Bigrset}{\Bigr\}}

\newcommand{\lip}{\left\langle}
\newcommand{\rip}{\right\rangle}
\newcommand{\bigglip}{\biggl\langle}
\newcommand{\biggrip}{\biggr\rangle}
\newcommand{\biglip}{\bigl\langle}
\newcommand{\bigrip}{\bigr\rangle}
\newcommand{\Biglip}{\Bigl\langle}
\newcommand{\Bigrip}{\Bigr\rangle}

\newcommand{\labs}{\left\vert}
\newcommand{\rabs}{\right\vert}
\newcommand{\biglabs}{\bigl\vert}
\newcommand{\bigrabs}{\bigr\vert}

\newcommand{\bigrest}{\big\vert}
\newcommand{\rest}{\vert}

\title[Spectral multipliers for the Kohn Laplacian]{Spectral multipliers for the Kohn Laplacian \\ on forms on the sphere in $\mathbb{C}^n$}
\author[V. Casarino]{Valentina Casarino}
\address{Dipartimento di Tecnica e Gestione dei Sistemi Industriali, Universit\`a degli Studi di Padova \\ Stradella San Nicola 3 \\ 36100 Vicenza \\ Italy}
\email{valentina.casarino@unipd.it}
\author[M.G. Cowling]{Michael G. Cowling}
\address{School of Mathematics\\ University of New South Wales\\ UNSW Sydney 2052\\ Australia}
\email{m.cowling@unsw.edu.au}
\author[A. Martini]{Alessio Martini}
\address{School of Mathematics\\ University of Birmingham\\Edgbaston\\Birmingham B15 2TT \\ UK}
\email{a.martini@bham.ac.uk}
\author[A. Sikora]{Adam Sikora}
\address{Department of Mathematics \\ Macquarie University \\ NSW 2109 \\ Australia}
\email{adam.sikora@mq.edu.au}
\thanks{The first and third-named authors were partially supported by GNAMPA (Progetto 2014 ``Moltiplicatori e proiettori spettrali associati a Laplaciani su sfere e gruppi nilpotenti'') and MIUR (PRIN 2010-2011 ``Variet\`a reali e complesse: geometria, topologia e analisi armonica'').
The second, third, and fourth-named authors were supported by the Australian Research Council (project DP110102488).
The third-named author gratefully acknowledges the support of the Humboldt Foundation and of the Deutsche Forschungsgemeinschaft (project MA 5222/2-1). }

\date{\today. Preliminary version}

\begin{document}
\begin{abstract}
The unit sphere $\Sphere$ in $\mathbb{C}^n$ is equipped with the tangential Cauchy--Riemann complex and the associated Laplacian $\boxb$.
We prove a H\"ormander spectral multiplier theorem for~$\boxb$ with critical index $n-1/2$, that is, half the topological dimension of $\Sphere$.
Our proof is mainly based on representation theory and on a detailed analysis of the spaces of differential forms on $\Sphere$.
\end{abstract}

\maketitle

\section{Introduction}

The sphere in $\C^n$ is often studied as a model CR manifold.
The tangential Cauchy--Riemann (CR) complex on the sphere and in the conformally equivalent context of the Heisenberg group was studied by various authors, including \cite{Fo,FoSt,Ge}.
The CR complex gives rise to a second order operator $\boxb$, of ``Laplace type'', which is sometimes sub\-elliptic and sometimes not.
This operator acts on $(i,j)$ forms, where $0 \leq i, j \leq n-1$,  but, like most authors, we restrict our attention to the case of $(0,j)$ forms.
With this restriction, since the sphere is strongly pseudoconvex, $\boxb$ is sub\-elliptic if $0 < j  < n-1$, but has an infinite-dimensional kernel when $j = 0$ or $j = n-1$ (see, for example, \cite{FoKo,KorVa}); this kernel may also be viewed as a CR-cohomology space.
In this paper, we deal with the cases when $\boxb$ is sub\-elliptic; in a future paper, we intend to deal with the remaining cases.
By doing so, we deal with forms in this paper, but manage to deal with functions only in the future one; on the other hand, we do not have to worry here about the complications such as the lack of Sobolev embedding theorems that arise from the fact that $\boxb$ has a nontrivial kernel.

Let $\AbsOp$ be a self-adjoint positive operator with dense domain in $L^2(M)$, the usual Lebesgue space of (equivalence classes of) functions on a $d$-dimensional manifold $M$, endowed with a measure that is absolutely continuous with respect to the Lebesgue measure when written in any coordinate system.
Then the spectral theorem allows us to form the bounded operator $F(\AbsOp)$ whenever $F$ is a bounded Borel function on $\R$, or just on the spectrum of $\AbsOp$.
We and many other authors seek conditions on $F$ that ensure that the operator $F(\AbsOp)$ extends continuously from $L^p(M) \cap L^2(M)$ to a bounded operator on $L^p(M)$ for some $p\neq 2$.

Let $H^s(\R)$ denote the usual Sobolev space on $\R$. 
We say that $F: \R \to \C$ satisfies a H\"ormander condition of order $s$ if
\begin{equation}\label{eq:Hormander-condition}
\sup_{t \in \R_+} \biglnorm  F(t \dummy  ) \, \eta \bigrnorm_{H^s(\R)} < \infty
\end{equation}
for one and hence all nonzero smooth functions $\eta$ with support in $[1,2]$; we fix one of these $\eta$.
In the case where $\Delta$ is the Laplacian $\AbsOp$ on $\R^d$, the H\"ormander multiplier theorem \cite{H} implies that $F(\AbsOp)$ is both bounded on $L^p(\R^d)$ when $1 < p < \infty$ and of weak type $(1,1)$ if the condition \eqref{eq:Hormander-condition} holds for some $s > d/2$.
We call $d/2$ the \emph{critical index}.
In this case, the critical index cannot be improved, and so this is a sharp multiplier theorem.

Sharp multiplier theorems have a long history, going back at least fifty years in the euclidean case.
When the operator $\AbsOp$ is sub\-elliptic but not elliptic, it is often associated to a \emph{homogeneous dimension} $Q$ that is larger than the topological dimension $d$ of $M$ (see, for example, \cite{FP,RoSt} for more about this). 
In many cases the euclidean techniques generalise reasonably readily to establish multiplier theorems with critical index $Q/2$; see, for example, \cite{Ch,MaMe} for the case of homogeneous sublaplacians on stratified Lie groups and \cite{He2,DOS} for the more general setting of spaces of homogeneous type. 
However in several examples it turns out that the critical index $Q/2$ is not sharp.

On the Heisenberg group and similar groups, sharp theorems with critical index $d/2$ were first proved by Hebisch \cite{He1} and M\"{u}ller and Stein \cite{MSt} for a homogeneous sublaplacian on functions; other step-two nilpotent groups have been treated by Martini and M\"{u}ller \cite{MarM}.
A corresponding result for the operator $\boxb$ on the Heisenberg group may be deduced from the work of M\"uller, Ricci and Stein \cite{MuRiSt}. 
There are various Laplacians associated to forms, and one may pose the same question for forms as for functions.
The case of the Hodge Laplacian on forms on the Heisenberg group was considered by M\"{u}ller, Peloso and Ricci \cite{MPeRi1, MPeRi2}.
At this time, it is unclear whether to expect that the sharp index is $d/2$ for all nilpotent Lie groups.

Cowling and Sikora \cite{CS} treated a sub\-elliptic operator $\AbsOp$ on functions on the group $\group{SU}(2)$, and these results were extended to a similar operator $\AbsOp$ on functions on the sphere in $\C^n$ by Cowling, Klima and Sikora \cite{CKS}; in all these theorems the critical index is $d/2$. 
The operator $\AbsOp$ is related to, but not the same as, the operator $\boxb$ on functions; in particular, $\AbsOp$ is sub\-elliptic while $\boxb$ is not. 
Note that, in the case of the sphere $\Sphere$, the topological dimension $d$ is $2n-1$, while the homogeneous dimension $Q$ associated to $\AbsOp$ and $\boxb$ is $2n$.

A quite general multiplier theorem for the operator  $\boxb$ acting on functions on a compact pseudoconvex CR manifold of finite type may be found in \cite{Street}. 
There the critical index is larger (it is equal to $(Q+1)/2$) and it is mentioned that similar methods yield an analogous result for the operator $\dbarba \dbarb$ acting on $(0,j)$-forms (which coincides with $\boxb$ in the case $j=0$).

For much more on the history of this kind of problem, see the references cited in the papers mentioned above.

Our main theorem may be stated as follows. 
Let $\Lambda^{0,j}$ denote the bundle of $(0,j)$-forms on $\Sphere$ and $L^p(\Lambda^{0,j})$ the corresponding space of $L^p$ sections.

\begin{theorem}\label{thm:main}
Let $\boxb$ be the Kohn Laplacian acting on $(0,j)$-forms on the unit sphere $\Sphere$ in $\C^n$, where $0 < j < n-1$.
Suppose that $s > n- 1/2$ and that $F: \R \to \C$ satisfies the H\"ormander condition \eqref{eq:Hormander-condition}.
Then $F(\boxb)$, initially defined on $L^2(\Lambda^{0,j})$, extends continuously to an operator on $L^p(\Lambda^{0,j})$ that is bounded when $1 < p < \infty$ and of weak type $(1,1)$.
Further, the associated operator norms are bounded by ($p$-dependent) multiples of $\sup_{t \in \R_+} \biglnorm  F(t \dummy  ) \, \eta \bigrnorm_{H^s(\R)}$.
\end{theorem}

In studying a Laplacian on a compact manifold, it is sometimes necessary to add the term $\labs m(0)\rabs$ to the bound on the operator norm, in order to take care of the zero spectrum. 
However in our case this is not necessary as the spectrum of $\boxb$ is strictly contained in $\R_+$.

The same methods allow us to obtain $L^p$-boundedness results for the Bochner--Riesz means associated to $\boxb$ on $\Sphere$.
Analogous results for the sublaplacian $\AbsOp$ on $\Sphere$ have been recently proved in \cite{CaPe}; we refer to \cite{Mauceri,Mueller} for earlier results on the Heisenberg group.
In all these papers, however, the convergence of the Bochner--Riesz means is proved only when $\delta> (2n-1) \labs {1}/{p} - {1}/{2} \rabs$.

\begin{theorem}\label{thm:bochnerriesz}
Let $\boxb$ be the Kohn Laplacian acting on $(0,j)$-forms on the unit sphere $\Sphere$ in $\C^n$, where $0 < j < n-1$. 
If $p \in [1,\infty]$ and
\[
\delta > (2n-2) \labs\frac{1}{p} - \frac{1}{2} \rabs,
\]
then the operators $(1-t \boxb)_+^\delta$ are bounded on $L^p(\Lambda^{0,j})$, uniformly in $t \in \R_+$.
\end{theorem}

The proof of these results is based on an abstract multiplier theorem of Cowling and Sikora \cite{CS}; strictly speaking, this needs to be adapted to deal with forms rather than functions, but this is a routine modification.
The theorem is also stated in \cite{CKS}, though that paper contains some minor errors which may lead to confusion; corrections are available from the first and third-named authors of \cite{CKS}.

The crucial step that allows us to obtain $n-1/2$ as critical index is the proof a ``weighted Plancherel estimate''. 
This is, roughly speaking, an estimate of a weighted $L^2$-norm of the integral kernel $K_{F(\boxb)}$ of the operator $F(\boxb)$ in terms of a (sort of) $L^2$-norm of the multiplier $F$. 
This estimate in turn reduces to the problem of determining how an eigenform of $\boxb$, after multiplication by a suitable weight, decomposes as a linear combination of eigenforms.

The operator $\boxb$ is $\group{U}(n)$-invariant, hence in order to determine its spectral resolution it is natural to consider the decomposition of the representation of $\group{U}(n)$ on $L^2(\Lambda^{0,j})$ into its irreducible components.
This decomposition was worked out by Folland \cite{Fo}, whose detailed analysis we use extensively and develop; in the case of functions ($j=0$), this is a refinement of the classical decomposition of $L^2(\Sphere)$ into spherical harmonics (see also \cite[Chapter 12]{Rudin}).
Since the representation on $L^2(\Lambda^{0,j})$ is multiplicity free, the operator $\boxb$ acts on each irreducible component as a scalar and forms associated to irreducible subrepresentations are eigenforms.

The key observation here is that the operation of multiplication of an eigenform by a (polynomial) weight may be interpreted in a representation-theoretic fashion, by taking the tensor product of an irreducible component of $L^2(\Lambda^{0,j})$ with a suitable representation of $\group{U}(n)$. 
Accordingly, the aforementioned decomposition of an eigenform multiplied by a weight corresponds to the decomposition of a tensor product representation into its irreducible components and classical results from the representation theory of $\group{U}(n)$ may be applied. 
In particular, the representations under consideration are multiplicity-free, and coupled with some symmetry properties of Clebsch--Gordan coefficients and the relations between the representations on $L^2(\Lambda^{0,j})$ for different values of $j$, this allows us to discover enough about the decompositions corresponding to multiplication by a weight to be able to prove the weighted Plancherel estimate.

Note that in \cite{CKS} a different route is followed: an explicit formula for the so-called zonal spherical functions is proved and used to determine the effect of multiplication by a weight.
No such formula is exploited here.
Hence, this paper provides an alternative approach to that of \cite{CKS} for the operator treated there; one should compare Lemma 3.1 of \cite{CKS} with our Theorem \ref{thm:Theorem-delta} (and their proofs).

The plan of the paper is the following.
The next section of this paper states the general multiplier theorem. 
Section 3 recalls the definition of the operator $\boxb$ and some of its basic properties, which immediately establish some of the conditions of the abstract theorem.
The detailed analysis of the spaces of forms on $\Sphere$ and the proof of the weighted Plancherel estimate are carried out in Sections 4 and 5.
In Section 6, we prove some representation theoretic results needed for our analysis.

\section{An abstract multiplier theorem}\label{section:abstracttheorem}

The proof of our result for the operator $\boxb$ is based on an abstract multiplier theorem obtained in \cite{CS} (namely, Theorem 3.6 there).

In fact, as stated in \cite{CS}, the theorem applies to operators acting on scalar (square-integrable) functions on a metric measure space $(X,\dist,\mu)$; that is, $X$ is a space with a metric (distance function) $\dist$ and a Borel measure $\mu$.
Here instead we work with operators acting on forms, that is, on sections of certain vector bundles over $X$.
However, the extension of the abstract theorem to the  case of vector bundles is straightforward, given the correct definitions and conventions (see, for instance, the remarks about this extension in \cite{Sik}).

Suppose that $\FE$ is a continuous complex vector bundle on $X$ of rank $m$ (that is, the fibres $\FE_x$ of $\FE$ are isomorphic to $\C^m$ for all $x \in X$) and with a measurable (with respect to $x$) inner product $\lip\dummy , \dummy \rip_x$ along the fibres. 
For $\alpha(x)\in \FE_x$ we put $\labs \alpha(x)\rabs_x^2=\lip \alpha(x), \alpha(x)\rip_x$. 
To simplify the notation,  we will often write $\lip\dummy , \dummy \rip$   and $\labs \dummy \rabs$ instead of $\lip\dummy , \dummy \rip_x$ and $\labs \dummy \rabs_x$.
Now for sections $\alpha$ and $\beta$ of $\FE$ we define $\lnorm\alpha\rnorm_{L^p(\FE)}$ and $\llangle \alpha,\beta \rrangle$ by
\[
\lnorm\alpha\rnorm_{L^p(\FE)}^p=\int_X\labs \alpha(x)\rabs^p \,d\mu(x) \qquad \text{and} \qquad  \llangle \alpha,\beta \rrangle= \int_X \lip \alpha(x),\beta(x)\rip \,d\mu(x).
\]
By $L^p(\FE)$ we denote the Banach spaces of sections of $\FE$ corresponding to these norms.
Note that $L^2(\FE)$ is a Hilbert space with the inner product $\left\llangle \dummy , \dummy  \right\rrangle$.

Next we describe the notion of \emph{integral operators}.
Suppose that $\FE$ and $\FF$ are continuous vector bundles of ranks $m$ and $n$ with base spaces $(X,\dist,\mu)$ and $(Y,\sigma,\nu)$ endowed with inner products as above.
Given $x \in X$ and $y \in Y$, we consider the space $\Hom(\FE_x,\FF_y)$ of all linear homomorphisms from $\FE_x$ to $\FF_y$.
We equip $\FE_x$ and $\FF_y$ with inner products and consider two natural norms on $\Hom(\FE_x,\FF_y)$: the Hilbert--Schmidt norm $\labs \dummy \rabs_{\HS}$ and the operator norm $\labs \dummy \rabs$.
Note that
\[
\labs K\rabs
\le  \labs K\rabs_{\HS}
\le \min(m,n)^{1/2}\labs K\rabs
\]
for all $K\in \Hom(\FE_x,\FF_y)$.
By $[\FE,\FF]$ we denote the continuous bundle with base space $Y \times X$ and with fibre $\Hom(\FE_x,\FF_y)$ over the point $(y,x)$ (note the change of order of $x$ and $y$ here).

We say that $T$ is an \emph{integral operator with kernel $K_{T}$} if $K_{T}$ is a section of  $[\FE,\FF]$ such that $\labs K_{T}\rabs$ is locally integrable on $(Y \times X, \nu\times\mu)$ and
\[
  \llangle T\alpha, \beta\rrangle
= \int_{Y} \lip T\alpha ,\beta\rip \, d \nu
= \int_{Y} \int_{X} \lip K_{T}(y,x) \, \alpha(x), \beta(y)\rip \, d \mu(x) \, d \nu(y)
\]
for all sections $\alpha$ in $C_c(\FE)$ and $\beta$ in~$C_c(\FF)$.

If $T$ is bounded from $L^1(\FE)$ to~$L^q(\FF)$, where $q > 1$, then $T$ is an integral operator, and
\[
\lnorm T\rnorm_{L^1(\FE) \to L^q(\FF)}
= \esssup_{x\in X} \sup_{\substack{v \in \FE_x \\ \labs v\rabs \leq 1}} \lnorm K_{T}(\dummy ,x) v \rnorm_{L^q(\FF)};
\]
conversely, if $T$ is an integral operator and the right hand side of the above equality is finite, then $T$ is bounded from $L^1(\FE)$ to~$L^q(\FF)$, even if $q = 1$. 
From the above equality, it follows in particular that
\[
{m}^{-1/q} \esssup_{x\in X} \lnorm \labs K_{T}(\dummy ,x)\rabs\rnorm_{L^q(Y)}
\le \lnorm T\rnorm_{L^1(\FE) \to L^q(\FF)}
\le \esssup_{x\in X} \lnorm \labs K_{T}(\dummy ,x)\rabs \rnorm_{L^q(Y)}.
\]
There is a dual characterization of the operator norm from $L^{q'}(\FE)$ to $L^\infty(\FF)$:
\begin{equation}\label{eq:operatornorm}
\begin{aligned}
{n}^{-1/q} \esssup_{y\in Y} \lnorm \labs K_{T}(y,\dummy )\rabs\rnorm_{L^{q}(X)}
&\le \lnorm T\rnorm_{L^{q'}(\FE) \to L^\infty(\FF)} \\
&\le\esssup_{y\in Y} \lnorm \labs K_{T}(y,\dummy )\rabs\rnorm_{L^{q}(X)}.
\end{aligned}
\end{equation}

Much as in \cite{CS}, for a Borel function $F$ supported in $[0,1]$, we define the norm  $\lnorm F\rnorm_{N,2}$ by the formula
\[
\lnorm F\rnorm_{N,2} 
= \biggl(\frac{1}{N}\sum_{i=1}^{N}  \sup_{\lambda\in [\frac{i-1}{N},\frac{i}{N}]}  \labs F(\lambda)\rabs^2\biggr)^{1/2},
\]
where $p \in  [1,\infty)$ and $N\in \Z_+$.
Now we can reformulate Theorem~3.6 of \cite{CS}.
In this statement, and elsewhere, the letter $C$ and variants such as $C_\ell$ denote constants, always assumed to be positive, which may vary from one occurrence to the next.
The expressions $a \simeq b$ and $a \lesssim b$ mean that there are constants $C$ and $C'$ such that $Ca \leq b \leq C'b$ and $a \leq C b$ respectively.

\begin{theorem}\label{thm:abstractmult}
Let $(X,\dist,\mu)$ be a bounded metric measure space, equipped with a weight function $\varpi\colon X \times X \to \R_+$, and let $d \in [1,\infty)$.
Let $\FE$ be a continuous vector bundle on $X$ with measurable inner product and $\AbsOp$ be a possibly unbounded positive self-adjoint operator with dense domain on $L^2(\FE)$. 
Suppose that the following hypotheses are verified:
\begin{enumerate}[(i)]
\item the doubling condition:
\[
\mu(B(x,2t)) \le C \,\mu(B(x,t)) \qquad\forall{x\in X} \quad\forall  t>0;
\]
\item the weighted estimate for balls:
\[
\int_{B(x,t)} \varpi(x,y)^{-1}\, d\mu(y) \le C \min(t^d,1);
\]
\item Sobolev-type estimates: for some sufficiently large integer $\ell$:
\[
\mu(B(x,t))^{1/2} \,\lnorm(1+t^2{\AbsOp})^{-\ell}\rnorm_{L^2(\FE) \to L^\infty(\FE)}\le C_\ell
\qquad\forall x \in X \quad\forall t\in \R_+;
\]
\item finite propagation speed:
\[
\supp \cos(t\sqrt{\AbsOp}) \alpha \subseteq \lset x \in X \colon \dist(x,\supp \alpha) \leq t\rset
\qquad\forall t \in \R_+ \quad\forall \alpha \in L^2(\FE);
\]
\item Plancherel-type estimates:
\[
\esssup_{y \in X} \biggl(\int_X \labs  K_{F(\sqrt{\AbsOp})}(x,y)\rabs^2 \, \varpi(x,y) \,d \mu(x) \biggr)^{1/2} 
\leq C \, N^{d/2} \lnorm F(N \dummy ) \rnorm_{N,2}
\]
for all $N \in \N$ and Borel functions $F$ such that $\supp F \subseteq [0,N]$.
\end{enumerate}
Finally, assume that $s>d/2$.
Then for all bounded Borel functions $F : \R \to \C$ such that
\[
\sup_{t\in\R_+}
\lnorm F(t\dummy ) \, \eta \rnorm_{H^s(\R)} < \infty,
\]
the operator $F(\sqrt{\AbsOp})$ is of weak type~$(1,1)$ and of strong type $(p,p)$ for all $p$ in~$(1,\infty)$; further, the associated operator norms are bounded by multiples of
\[
\sup_{t\in\R_+} \lnorm F(t \dummy ) \, \eta \rnorm_{H^s(\R)}
        + \labs  F(0) \rabs .
\]
\end{theorem}

It is perhaps worth noting that Hypotheses (iii) and (iv) amount to ``on-diagonal'' and ``off-diagonal Gaussian'' estimates for the heat kernel associated to $\boxb$.

An inspection of the proof of the above theorem (see, in particular, page 26 of \cite{CS}), shows that an $L^1$-boundedness result may be obtained under similar assumptions in the case of compactly supported multipliers.

\begin{theorem}\label{thm:abstractmult2}
Assume Hypotheses (i)  to (v) of the previous theorem.
If $s > d/2$, then for all $t \in \R_+$ and all bounded Borel functions $F : \R \to \C$ with $\supp F \subseteq [-1,1]$,
\[
\esssup_{y \in X} \int_X \labs K_{F(t \sqrt{\AbsOp})}(x,y)\rabs \, dx  \leq C_s \lnorm F \rnorm_{H^s(\R)};
\]
consequently, if $\lnorm F \rnorm_{H^s(\R)} < \infty$, then for all $p \in [1,\infty]$ the operator $F(t\sqrt{\AbsOp})$ is bounded on $L^p(\FE)$, uniformly in $t \in \R_+$, and
\[
\sup_{t \in \R_+} \lnorm F(t\sqrt{\AbsOp})\rnorm_{L^p(\FE)\to L^p(\FE)} \leq C_s \lnorm F\rnorm_{H^s(\R)}.
\]
\end{theorem}

\section{The tangential Cauchy--Riemann complex on the sphere}\label{section:complex}

Fix $n \geq 2$.
As a real hypersurface in $\C^n$, the unit sphere $\Sphere$ is naturally endowed with a CR structure (of hypersurface type). 
Namely, let $\C T \C^n = T_{1,0} \C^n \oplus T_{0,1} \C^n$ be the decomposition of the complexified tangent bundle of $\C^n$ into its holomorphic and antiholomorphic components. 
Then the definition
\[
\CRL = \C T \Sphere \cap T_{1,0} \C^n
\]
gives an involutive subbundle $\CRL$ of rank $n-1$ of the complexified tangent bundle $\C T \Sphere$; moreover $\Bar{\CRL} = \C T \Sphere \cap T_{0,1} \C^n$ and
\[
\Bar{\CRL}_w \cap \CRL_w = \lset 0\rset
\qquad\forall w \in \Sphere.
\]

The dual bundle $\Bar{\CRL}^*$ of $\Bar{\CRL}$ is identified with a subbundle of the complexified cotangent bundle $\C T^* \Sphere$ via the standard Riemannian metric on $T \Sphere$ induced by $\C^n$.
Correspondingly the $j$th exterior power $\Lambda^{0,j} = \Lambda^j \Bar{\CRL}^*$ may be identified with a subbundle of the bundle $\C \Lambda^j \Sphere = \Lambda^j \C T^* \Sphere$ of $j$-forms on $\Sphere$; in particular, the space of sections of $\Lambda^{0,j}$ may be viewed as a subspace of the space of $j$-forms on $\Sphere$, and the fibrewise orthogonal projection defines a bundle morphism $\pi_j : \C \Lambda^j \Sphere \to \Lambda^{0,j}$. 
Consequently the definition
\[
\dbarb \alpha = \pi_{j+1} d \alpha,
\]
where $\alpha$ is a section of $\Lambda^{0,j}$ and $d$ is the exterior derivative, gives rise to a first-order differential operator $\dbarb : C^\infty(\Lambda^{0,j}) \to C^\infty(\Lambda^{0,j+1})$. 
One may prove that $\dbarb^2 = 0$ (see, for example, \cite[Section 8.2]{Boggess}) and the complex that arises, namely,
\[
0 \longrightarrow C^\infty(\Lambda^{0,0}) \stackrel{\dbarb}{\longrightarrow} C^\infty(\Lambda^{0,1}) \stackrel{\dbarb}{\longrightarrow} \dots \stackrel{\dbarb}{\longrightarrow} C^\infty(\Lambda^{0,n-2}) \stackrel{\dbarb}{\longrightarrow} C^\infty(\Lambda^{0,n-1}) \longrightarrow 0,
\]
is known as the tangential Cauchy--Riemann complex on $\Sphere$.
Since the group $\group{U}(n)$ acts on the sphere $\Sphere$ via restrictions of maps which are both isometric and holomorphic, the action of $\group{U}(n)$ preserves the Riemannian metric, the standard surface measure $\sigma$, the CR structure $\Bar{\CRL}$, and the corresponding complex, and $\dbarb$ is $\group{U}(n)$-equivariant.

Associated to this complex, we define the formal adjoint $\dbarba$ of $\dbarb$ with respect to the inner product
\[
\llangle \alpha, \beta \rrangle = \int_{\Sphere} \lip \alpha(z), \beta(z) \rip \,d\sigma(z),
\]
and the second-order operator $\boxb$ by
\[
\boxb = (\dbarb + \dbarba)^2 = \dbarb \dbarba + \dbarba \dbarb.
\]
The maps $\dbarba$ and $\boxb$ are also $\group{U}(n)$-equivariant.
We set $\Forms{j} = L^2(\Lambda^{0,j})$; equivalently, $\Forms{j}$ is the Hilbert space completion of $C^\infty(\Lambda^{0,j})$ with respect to this inner product.

The Riemannian distance on $\Sphere$ is not appropriate for analysis of operators such as $\dbarb$, $\dbarba$ and $\boxb$.
To these operators we may associate a control distance $\dist_0$ (see, for example, \cite[Section 8.4]{CM}).
This distance is $\sqrt{2}$ times the subriemannian distance on $\Sphere$ defined by taking as horizontal distribution $H$ the Levi distribution $(\CRL \oplus \Bar{\CRL}) \cap T \Sphere$ and endowing it with the restriction of the Riemannian inner product on $T\Sphere$. 
Note that
\[
H_w = \lset z \in \C^n \colon \lip z, w \rip = 0 \rset 
\subseteq \lset z \in \C^n \colon \Re \lip z, w \rip = 0 \rset = T_w \Sphere
\]
for all $w \in \Sphere$, where $\lip \dummy ,\dummy  \rip$ denotes the usual Hermitian inner product on $\C^n$, that is,
\[
\lip z,w \rip = \sum_{m=1}^n z_m \Bar{w}_m,
\]
and $H$ and hence also $\dist_0$ are $\group{U}(n)$-invariant.

Since the CR manifold $\Sphere$ is strictly pseudoconvex (that is, the Levi form is non\-degenerate, see \cite[Chapter 10]{Boggess}), the distribution $H$ is bracket-generating of step~$2$. 
Hence, by Chow's theorem, the associated subriemannian distance is finite and induces the standard topology on $\Sphere$.
Moreover the metric space $(\Sphere,\dist_0)$ is compact and hence complete, so the operator $\boxb$ is essentially self-adjoint and satisfies the finite propagation speed property
\begin{equation}\label{eq:fps}
\supp \cos(t\sqrt{\boxb}) f \subseteq \lset x \in \Sphere \colon \dist_0(x,\supp f) \leq t\rset
\qquad\forall t \in \R_+ \quad\forall f \in \Forms{j}
\end{equation}
(see \cite[Section 7]{CM}).

An explicit expression for the control distance $\dist_0$ is difficult to obtain or to work with (see \cite{CMV,BaW}), and so it is convenient to use an equivalent distance $\dist$, given by
\begin{equation}\label{eqn:defn-of-dist}
\dist(z,w) = 2  \, \labs  1 - \lip z,w \rip \rabs^{1/2}
\end{equation}
for all $z,w \in \Sphere$; this distance is evidently $\group{U}(n)$-invariant. 
For more on $\dist$, including a proof of the triangle inequality, see, for instance, \cite[Section 5.1]{Rudin}.

\begin{proposition}
The distance functions $\dist_0$ and $\dist$ are equivalent; more precisely, there is a constant $C$ such that
\[
\dist(z,w) \leq \dist_0(z,w) \leq C \dist(z,w)
\qquad\forall w, z \in \Sphere.
\]
\end{proposition}
\begin{proof}
As the distances $\dist$ and $\dist_0$ are $\group{U}(n)$-invariant, we may suppose temporarily without loss of generality that $w = (1, 0, \dots, 0)$; it is then convenient to write $z = (z_1, z')$, where $z' = (z_2, \dots, z_n)$.
Now 
\begin{equation}\label{eqn:distance-comparison}
\dist_0(z, w) \simeq \labs z'\rabs + \labs z_1\rabs^{1/2} \simeq \dist(z, w) 
\end{equation}
for all $z \in \Sphere$; the left-hand equivalence follows from the ball-box theorem (see \cite{NSW} or  \cite[Theorem 7.34]{Bellaiche}), while the right-hand equivalence follows from the definition \eqref{eqn:defn-of-dist} by computation.

It remains to show that $\dist(z,w) \leq \dist_0(z,w) $, that is, to determine one of the constants in the equivalence.
First of all, we take $z = (1, 0, \dots, 0)$ and show that $\labs \dbarb \dist(z,w) \rabs \leq 1$ for all $w \in \Sphere \setminus\lset z\rset$, where $\dbarb$ acts in the first variable.  
Observe that, for this $z$,
\begin{equation*}
\dbarb \dist(z,w) 
= \sum_{m=2}^n \frac{\partial \dist(z,w)}{\partial \Bar{z}_m} \,d\Bar{z}_m  \,;
\end{equation*}
since $\labs  d\Bar{z}_m \rabs = \sqrt{2}$, it follows that
\[
\begin{aligned}
\labs \dbarb \dist(z,w) \rabs
&= \sqrt{2} \biggl(\sum_{m=2}^{n} \labs\frac{\partial}{\partial \Bar{z}_m} 2 (1 - \lip w,z \rip) ^{1/4} (1 - \lip z,w \rip)^{1/4} \rabs^2 \biggr)^{1/2} \\
&= \frac{1}{\sqrt{2}} \biggl( \sum_{m=2}^{n} \frac{ \labs w_m\rabs^2}{ \labs 1 - \lip z,w \rip\rabs } \biggr)^{1/2} \\
&= \frac{1}{\sqrt{2}} \left( \frac{ (1 - \labs w_1\rabs)(1+ \labs w_1\rabs) }{ \labs 1 - w_1\rabs } \right)^{1/2} 
\leq 1 .
\end{aligned}
\]	
This inequality holds for all $z \in \Sphere$ and all $w \in \Sphere \setminus \lset z\rset$ by $\group{U}(n)$-invariance.
It follows from \cite[Proposition 5.4]{CM} that $\dist({}\cdot{}, w) \leq \dist_0({}\cdot{}, w)$.
\end{proof}

The following result is an immediate consequence of this distance function comparison and \eqref{eq:fps2}.

\begin{corollary}
The operator $\boxb$ has the finite propagation speed property relative to $\dist$:	
\begin{equation}\label{eq:fps2}
\supp \cos(t\sqrt{\boxb}) f \subseteq \lset x \in \Sphere \colon \dist(x,\supp f) \leq t\rset
\qquad\forall t \in \R_+ \quad\forall f \in \Forms{j} .
\end{equation}
\end{corollary}

\section{Harmonic analysis of forms}

Recall that the differential operators $\dbarb$, $\dbarba$, and $\boxb$ are $\group{U}(n)$-equivariant on the various spaces of forms on the sphere, that is, they are intertwining operators between the natural representations of $\group{U}(n)$ on the various spaces $\Forms{j}$.
Therefore it is possible to study the spectral properties of these operators through the unitary representation theory of $\group{U}(n)$.

This route was followed by Folland in \cite{Fo}. 
Recall that irreducible unitary representations of $\group{U}(n)$ may be parametrized (modulo equivalence) by nonincreasing $n$-tuples $(\ell_1,\dots,\ell_n) \in \Z^n$; as in \cite{Fo} we will denote the corresponding representation by $\rho(\ell_1,\dots,\ell_n)$. 
Folland determined the orthogonal direct sum decomposition of the spaces $\Forms{j}$ into irreducible subspaces and identified the representations that appear, as follows.

\begin{proposition}[{\cite[Theorem 2]{Fo}}]
Let $0 \leq j  \leq n-1$.
Then
\[
\Forms{j} = 
\begin{cases}
 \bigoplus_{p\geq 0,q\geq 0} \Phi_{pq0} & \text{when $j=0$} \\[3pt]
  \bigoplus_{p\geq 0,q\geq 1} (\Phi_{pqj} \oplus \Psi_{pqj}) & \text{when $1 \leq j \leq n-2$} \\[3pt]
   \bigoplus_{p\geq -1,q\geq 1}  \Psi_{pq(n-1)} &  \text{when $j= n-1$ .}
\end{cases} 
\]
The subspaces $\Phi_{pqj}$ and  $\Psi_{pqj}$ correspond to the two irreducible unitary representations $\rho(q,\underline{1}_j,\underline{0}_{n-2-j},-p)$ and $\rho(q,\underline{1}_{j-1},\underline{0}_{n-1-j},-p)$ of $\group{U}(n)$. 
In particular, the action of $\group{U}(n)$ on $\Forms{j}$ is multiplicity-free, that is, no irreducible unitary representation of $\group{U}(n)$ occurs more than once in $\Forms{j}$.
\end{proposition}

Note that the symbols $\underline{0}_\ell$ and $\underline{1}_\ell$ denote $\ell$ consecutive entries of $0$ or $1$.
This decomposition leads us to introduce the following index sets:
\begin{align*}
I_j& := \begin{cases}
\lset(p,q)\,:\,
p\ge 0, q\ge 0 \rset  & \text{when $j=0$,} \\
\lset(p,q)\,:\,
p\ge 0, q\ge 1
\rset  & \text{when $1\le j\le n-2$,} \\
\lset(p,q)\,:\,
 p\ge -1, q\ge 1
\rset  & \text{when $j=n-1$;}
\end{cases}
\\
\noalign{\noindent{and}}
\TT_j     &:=
\begin{cases}
\lset\Phi\rset   & \text{when $j=0$,} \\
\lset\Phi,\Psi\rset & \text{when $1\le j\le n-2$,} \\
\lset\Psi\rset  & \text{when $j=n-1$.}
\end{cases}
\end{align*}
Then the set $I_j \times Y_j$ parametrises the irreducible representations of $\group{U}(n)$ that appear in the decomposition of the representation of $\group{U}(n)$ on $\Forms{j}$.
We abuse notation slightly and write $(p,q,\Upsilon) \in I_j \times Y_j$ to mean that $((p,q),\Upsilon) \in I_j \times Y_j$.
When $(p,q,\Upsilon) \in I_j \times Y_j$, we write $\Upsilon_{pqj}$ for one of  the spaces $\Phi_{pqj}$ or $\Psi_{pqj}$, depending on whether $\Upsilon=\Phi$ or $\Upsilon = \Psi$.
We adopt the convention that $\Upsilon_{pqj} = \lset0\rset$ when $(p,q,\Upsilon) \notin I_j \times Y_j$.

The subspaces $\Phi_{pqj}$ and $\Psi_{pqj}$ are finite-dimensional spaces of smooth forms (see below), hence they lie in the domain of all smooth differential operators. 
Since the representation of $\group{U}(n)$ on each $\Forms{j}$ is multiplicity-free and the operators $\dbarb$, $\dbarba$, and $\boxb$ are $\group{U}(n)$-equivariant, these operators must preserve the above decomposition; more precisely, they must map irreducible components into equivalent irreducible components by multiples of unitary operators.
To complete the picture, Folland determined these multiples, that is, he computed the ``eigenvalues'' of $\dbarb$, $\dbarba$, and $\boxb$ on each piece of the decomposition above.

\begin{proposition}[{\cite[Theorems 4 and 6]{Fo}}]\label{prp:eigenvalues}
Let $0\le j\le n-2$ and define
\begin{equation}\label{eq:deflambdapqj}
\lambda_{pqj}
=\bigl(2{(q+j)}{(p+n-1-j)}\bigr)^{1/2}\,,
\end{equation}
for all $(p,q)\in I_j \cup I_{j+1}$.
Then
\begin{enumerate}[(i)]
\item
$\dbarb(\Phi_{pqj}) = \Psi_{pq(j+1)}$ for all $(p,q)\in I_j$,
and $\dbarb\rest_{\Phi_{pqj}}$ is $\lambda_{pqj}$ multiplied by  a unitary operator;
\item $\dbarb(\Psi_{pq(j+1)}) = \lset0\rset$ for all $(p,q)\in I_{j+1}$;
\item $\dbarba(\Psi_{pq(j+1)}) = \Phi_{pqj}$
for all $(p,q)\in I_{j+1}$,
and $\dbarba\rest_{\Psi_{pq(j+1)}}$ is $\lambda_{pqj}$ multiplied by  a unitary operator;
\item $\dbarba(\Phi_{pqj}) =\lset0\rset$ for all $(p,q)\in I_{j}$;
\item $\boxb(\Phi_{pqj}) \subseteq \Phi_{pqj}$ and $\boxb\rest_{\Phi_{pqj}} = \lambda_{pqj}^2 \id_{\Phi_{pqj}}$ for all $(p,q)\in I_j$;
\item $\boxb(\Psi_{pq(j+1)}) \subseteq \Psi_{pq(j+1)}$ and $\boxb\rest_{\Psi_{pq(j+1)}}
= \lambda_{pqj}^2 \id_{\Psi_{pq(j+1)}}$
for all $(p,q)\in  I_{j+1}$.
\end{enumerate} 
\end{proposition}

The spectral resolution of $\boxb$ on $\Forms{j}$ may now be expressed in terms of the orthogonal projections $P^{\Phi}_{pqj}$ and $P^{\Psi}_{pqj}$ in $\Lin(\Forms{j})$ onto the subspaces $\Phi_{pqj}$ and $\Psi_{pqj}$.
Let $K^\Phi_{pqj}$ and $K^\Psi_{pqj}$ be the corresponding integral kernels, and  set  $\lambda_{pqj}^\Phi = \lambda_{pqj}$ and $\lambda_{pqj}^\Psi = \lambda_{pq(j-1)}$.

\begin{lemma}\label{lem:Cor-4.3v}
Let $0\le j\le n-1$.
Suppose that $F : \R \to \C$ is a compactly supported Borel function.
Suppose further that $F(0)=0$ when $j=0$ or
$j=n-1$.
Then
the kernel $K_{F(\sqrt{\boxb})}$
of the operator $F(\sqrt{\boxb})$
is given
by the formula
\begin{equation}\label{eq:kernelformula}
K_{ F(\sqrt{\boxb})}(z,w)=\sum_{(p,q,\Upsilon) \in I_j \times Y_j} F({\lambda^\Upsilon_{pqj}}) \, K^{\Upsilon}_{pqj}(z,w)
\quad\forall z,w \in \Sphere.
\end{equation}
\end{lemma}

Note that, under the assumptions of Lemma \ref{lem:Cor-4.3v}, the sum in \eqref{eq:kernelformula} has a finite number of nonzero summands. 
In fact, \eqref{eq:kernelformula} holds also for functions $F$ with noncompact support that decay sufficiently rapidly at infinity. 
This is an easy consequence of the following orthogonality relations.

\begin{proposition}\label{prp:HS-projection}
Let $0 \leq j \leq n-1$.
For all $(p,q,\Upsilon) ,(p',q',\Upsilon') \in I_j \times Y_j$, 
\[
\int_{\Sphere}\biglip K^{\Upsilon}_{pqj}(z,w), K^{\Upsilon'}_{p'q'j} (z,w)\bigrip_{\HS}\, d\sigma(z)
=\frac{\dim \Upsilon_{pqj}}{\sigma(\Sphere)} \, \delta_{\Upsilon\Upsilon'}\delta_{pp'}\delta_{qq'}
\qquad\forall w \in \Sphere.
\]
\end{proposition}

\begin{proof}
We assume that $1 \leq j \leq n-2$; the other cases are similar but easier.
Let $\lset  \varphi^r_{pqj}\,:\, 1\le r\le \dim \Phi_{pqj}  \rset$ and $\lset  \psi^s_{pqj}\,:\, 1\le s\le \dim \Psi_{pqj}  \rset$ be ortho\-normal bases of $\Phi_{pqj}$ and $\Psi_{pqj}$.
Then, for all $z,w \in \Sphere$,
\begin{gather*}
K^\Phi_{pqj} (z,w)
= \sum_{r=1}^{\dim \Phi_{pqj}} \lip \dummy , \varphi^r_{pqj}(w) \rip \,\varphi^r_{pqj}(z)\,,\\
K^\Psi_{pqj} (z,w)
= \sum_{s=1}^{\dim \Psi_{pqj}} \lip \dummy , \psi^s_{pqj}(w) \rip \,\psi^s_{pqj}(z)\,.
\end{gather*}
Then
\[\begin{aligned}
&\int_{\Sphere}\biglip K^{\Phi}_{pqj}(z,w), K^{\Phi}_{p'q'j} (z,w)\bigrip_{\HS}\, d\sigma(z)\\
&\qquad= \int_{\Sphere} \sum_{r=1}^{\dim \Phi_{pqj}} \sum_{s=1}^{\dim \Phi_{p'q'j}} \Biglip \lip \dummy , \varphi^r_{pqj}(w) \rip \,\varphi^r_{pqj}(z)\,, \lip \dummy , \varphi^s_{p'q'j}(w) \rip \,\varphi^s_{p'q'j}(z)\, \Bigrip_{\HS}\, d\sigma(z)
\\
&\qquad= \int_{\Sphere} \sum_{r=1}^{\dim \Phi_{pqj}} \sum_{s=1}^{\dim \Phi_{p'q'j}} \biglip  \varphi^s_{p'q'j}(w), \varphi^r_{pqj}(w) \bigrip  \, \biglip \varphi^r_{pqj}(z) \,,\varphi^s_{p'q'j}(z)\, \bigrip\, d\sigma(z)
\\
&\qquad = \sum_{r=1}^{\dim \Phi_{pqj}} \sum_{s=1}^{\dim \Phi_{p'q'j}} \biglip  \varphi^s_{p'q'j}(w), \varphi^r_{pqj}(w) \bigrip \, \left\llangle \varphi^r_{pqj} ,\varphi^s_{p'q'j}\right\rrangle
\\
&\qquad = \delta_{rs} \delta_{pp'}\delta_{qq'} \sum_{r=1}^{\dim \Phi_{pqj}} \biglip  \varphi^r_{pqj}(w), \varphi^r_{pqj}(w) \bigrip .
\end{aligned}\]
Since this equality holds for all orthonormal bases $\lset  \varphi^r_{pqj}\,:\, 1\le r\le \dim \Phi_{pqj}  \rset$ of $\Phi_{pqj}$ and $\group{U}(n)$ maps orthonormal bases to orthonormal bases, the last expression must be independent of $w$. 
In particular, by averaging over the sphere, we obtain
\[\begin{aligned}
&\int_{\Sphere}\biglip K^{\Phi}_{pqj}(z,w), K^{\Phi}_{p'q'j} (z,w)\bigrip_{\HS}\, d\sigma(z)\\
&\qquad=
\frac{\delta_{rs} \delta_{pp'} \delta_{qq'}}{\sigma(\Sphere)} \sum_{r=1}^{\dim \Phi_{pqj}}  \int_{\Sphere}  \biglip \varphi^r_{pqj}(w), \varphi^r_{pqj}(w) \bigrip \,d\sigma(w) \\
&\qquad=
\frac{\delta_{rs} \delta_{pp'} \delta_{qq'}}{\sigma(\Sphere)} \sum_{r=1}^{\dim \Phi_{pqj}} \llangle \varphi^r_{pqj}, \varphi^r_{pqj} \rrangle  \\
&\qquad=
\frac{\delta_{rs} \delta_{pp'} \delta_{qq'}}{\sigma(\Sphere)} \dim \Phi_{pqj},
\end{aligned}\]
for all $w \in \Sphere$, proving the case where $\Upsilon=\Upsilon'=\Phi$.
The case where $\Upsilon=\Upsilon'=\Psi$ may be treated analogously, while the case where $\Upsilon \neq \Upsilon'$ follows from the orthogonality of the system $\lset  \varphi^r_{pqj}\,, \psi^s_{pqj}\,:\, 1\le r\le \dim \Phi_{pqj} \,,\,\, 1\le s\le \dim \Psi_{pqj}  \rset$.
\end{proof}

The main result of this section is the following theorem; it should be compared with Lemma 3.1 and Corollary 3.2 in \cite{CKS}.

\begin{theorem}\label{thm:Theorem-delta}
Let $0 \leq j \leq n-1$.
For all $(p,q,\Upsilon) \in I_j \times Y_j$,
\begin{align}
\overline{\lip z,w \rip}\, K^\Upsilon_{pqj}(z,w) 
&= \sum_{(p',q',\Upsilon') \in I_j \times Y_j} \Bar{\delta}^{\Upsilon \Upsilon'}_{pp'qq'j} K^{\Upsilon'}_{p'q'j}(z,w) ,
\label{eq:deltabar} \\
\lip z,w \rip K^\Upsilon_{pqj}(z,w) 
&= \sum_{(p',q',\Upsilon') \in I_j \times Y_j} \delta^{\Upsilon \Upsilon'}_{pp'qq'j} K^{\Upsilon'}_{p'q'j}(z,w) , \label{eq:delta}
\end{align}
where
\begin{align*}
\delta^{\Psi\Psi}_{p(p+1)qqj} &= \frac{p+n+1-j}{p+q+n} \cdot \frac{p+1}{p+n-j}  \,,
&
\Bar{\delta}^{\Phi\Phi}_{ppq(q+1)j} &= \frac{q+1+j}{p+q+n} \cdot \frac{q}{q+j}\,, 
\\
\delta^{\Psi\Psi}_{ppq(q-1)j} &=  \frac{q+n-2}{p+q+n-2} \cdot \frac{q-2+j}{q-1+j}\,,
&
\Bar{\delta}^{\Phi\Phi}_{p(p-1)qqj} &= \frac{p+n-1}{p+q+n-2} \cdot \frac{p+n-2-j}{p+n-1-j} \,,
\\
\delta^{\Phi\Phi}_{ppq(q-1)j} &=\frac{q+n-2}{p+q+n-2}\,,
&
\Bar{\delta}^{\Psi\Psi}_{p(p-1)qqj} &=\frac{p+n-1}{p+q+n-2}\,,
\\
\delta^{\Phi\Phi}_{p(p+1)qqj} &= \frac{p+1}{p+q+n} \,, 
&
\Bar{\delta}^{\Psi\Psi}_{ppq(q+1)j} &= \frac{q}{p+q+n}\,,
\\
\delta^{\Phi\Psi}_{ppqqj} &= \frac{n-1-j}{(q+j)(p+n-1-j)},
&
\Bar{\delta}^{\Psi\Phi}_{ppqqj} &= \frac{j}{(q-1+j)(p+n-j)}\,,
\end{align*}
while all the other coefficients vanish.
\end{theorem}

\begin{remark}\label{rem:symmetry}
Note that the coefficients in Theorem \ref{thm:Theorem-delta} are all real; the bar does not indicate complex conjugation.

The fractions $q/(q+j)$ and $(p+1)/(p+n-j)$ are interpreted as $1$ when they are of the form $0/0$.

The right-hand side fractions in each row are symmetric; one becomes the other when we  exchange $p$ with $q-1$ and $j$ with $n-1-j$.
This may be explained using representation theory, as follows.

In analogy with the Hodge star operator, one may define an isometric antilinear vector bundle isomorphism $*$ from $\Lambda^{0,j}$ to $\Lambda^{0,n-1-j}$ (see \cite[p.~5]{Ge}), which induces an operator, also denoted by $*$, from $\Forms{j}$ to $\Forms{n-1-j}$.
It is not difficult to check that $*$ intertwines:
\begin{enumerate}[(i)]
\item the standard representation $\pi_j$ of $\group{U}(n)$ on  $\Forms{j}$ and the representation $\hat\pi_{n-1-j}$ of $\group{U}(n)$ on $\Forms{n-1-j}$ given by
\[
\hat\pi_{n-1-j}(g) = (\det g)^{-1} \pi_{n-1-j}(g) \quad\text{for all $g \in \group{U}(n)$;}
\]
\item the operators $\boxb$ on $\Forms{j}$ and $\boxb$ on $\Forms{n-1-j}$ \cite[Lemma 1.1]{Ge};
\item the multiplication operators $M_m,\Bar{M}_m$ on $\Forms{j}$ and $\Bar{M}_m,M_m$ on $\Forms{n-1-j}$, where
\[
M_m f(z) = z_m f(z), \qquad \Bar{M}_m f(z) = \Bar{z}_m f(z).
\]
\end{enumerate}
If $J : \Forms{j} \to (\Forms{j})^*$ is the canonical antilinear isomorphism between a Hilbert space and its dual, then the composition $* J^{-1} : (\Forms{j})^* \to \Forms{n-1-j}$  is a linear isomorphism that intertwines the representation $\tilde\pi_j$ contragredient to $\pi_j$ and the representation $\hat\pi_{n-1-j}$.
Hence, by (i) and Schur's lemma, 
\begin{equation}\label{eqn:Hodge-star-on-Upsilon}
{*}(\Phi_{pqj}) = \Psi_{(q-1)(p+1)(n-1-j)}
\qquad\text{and}\qquad
{*}(\Psi_{pqj}) = \Phi_{(q-1)(p+1)(n-1-j)}
\end{equation}
for all $(p,q) \in I_j$.
Consequently (ii) justifies the symmetry
\begin{equation}\label{eq:eigenvaluesymmetry}
\lambda_{pqj} = \lambda_{(q-1)(p+1)(n-2-j)}
\end{equation}
of the eigenvalues of $\boxb$. 
Moreover, $\lip z,w \rip K_{pqj}^\Upsilon(z,w)$ and $\overline{\lip z,w \rip} K_{pqj}^\Upsilon(z,w)$ are the integral kernels of the operators $\sum_{m=1}^{n} M_m P^\Upsilon_{pqj} \Bar{M}_m$ and $\sum_{m=1}^{n} \Bar{M}_m P^\Upsilon_{pqj} M_m$, and so (iii) and \eqref{eqn:Hodge-star-on-Upsilon} lead to the conclusion that 
\[
\delta_{pp'qq'j}^{\Upsilon \Upsilon'} = \Bar{\delta}_{(q-1)(q'-1)(p+1)(p'+1)(n-1-j)}^{\Bar{\Upsilon} \Bar{\Upsilon}'} \,,
\]
where $\Bar{\Phi} = \Psi$ and $\Bar{\Psi} = \Phi$.
\end{remark}

The proof of Theorem \ref{thm:Theorem-delta} is based on the following preliminary result.

\begin{proposition}\label{prp:TheoremProjectionv}
Let $0 \leq j \leq n-1$.
The decompositions \eqref{eq:deltabar} and \eqref{eq:delta} hold, and
\begin{align*}
\Bar{\delta}_{pp'qq'j}^{\Upsilon\Upsilon'}
&=
\frac{\dim \Upsilon_{pqj}}{\dim \Upsilon'_{p'q'j}} \sum_{m=1}^{n} \frac{\lnorm P_{p'q'j}^{\Upsilon'} (\Bar{z}_m \alpha) \rnorm^2}{\lnorm\alpha\rnorm^2} \,, \\
\noalign{\noindent{and}}
\delta_{pp'qq'j}^{\Upsilon\Upsilon'}
&= \frac{\dim \Upsilon_{pqj}}{\dim \Upsilon'_{p'q'j}} \sum_{m=1}^{n} \frac{\lnorm P_{p'q'j}^{\Upsilon'} (z_m \alpha)  \rnorm^2} {\lnorm\alpha\rnorm^2},
\end{align*}
for all $(p,q,\Upsilon), (p',q',\Upsilon') \in I_j \times Y_j$; here $\alpha$ is any nonzero element of $\Upsilon_{pqj}$.
\end{proposition}
\begin{proof}
Let $\Upsilon',\Upsilon'' \in \TT_j$. 
For given $\alpha' \in \Upsilon'_{p'q'j}$ and $\alpha'' \in \Upsilon''_{p''q''j}$, define the operator $P_{\alpha',\alpha''} \in \Lin(\Forms{j})$ by $P_{\alpha',\alpha''}:= \lip \dummy  , \alpha'\rip \alpha''$.
Consider the average $\tilde P_{\alpha',\alpha''}$ over $\group{U}(n)$ of $P_{\alpha',\alpha''}$, that is,
\[
\tilde P_{\alpha',\alpha''} := \int_{\group{U}(n)} \pi (g)\,P_{\alpha',\alpha''} \, \pi (g)^{-1} \,dg\,,
\]
where $\pi$ is the representation of $\group{U}(n)$ on $\Forms{j}$ and $dg$ denotes the normalized Haar measure on $\group{U}(n)$. 
Then $\tilde P_{\alpha',\alpha''} \in \Lin(\Forms{j})$ is an intertwining operator; moreover $\tilde P_{\alpha',\alpha''}(\Upsilon'_{p'q'j}) \subseteq \Upsilon''_{p''q''j}$ and $\tilde P_{\alpha',\alpha''}\rest_{{\Upsilon'_{p'q'j}}^\perp} = 0$.  
Hence from Schur's lemma 
\[
\tilde P_{\alpha',\alpha''} = 
\begin{dcases}
\frac{\lip \alpha'', \alpha'\rip}{\dim \Upsilon'_{p'q'j}}
 P^{\Upsilon'}_{p'q'j} &\text{when $\Upsilon_{pqj} = \Upsilon'_{p'q'j}$,}\\
 0 &\text{otherwise,}
\end{dcases}
\]
where the first identification relies on the facts that $\tr \tilde P_{\alpha',\alpha''} = \tr P_{\alpha',\alpha''} = \lip \alpha'',\alpha'\rip$ and $\tr P^{\Upsilon'}_{p'q'j} = \dim \Upsilon'_{p'q'j}$.
Correspondingly, let $K_{\alpha',\alpha''}$ denote the kernel of the operator $\tilde P_{\alpha',\alpha''}$; then
\begin{equation}  \label{eq:relazione-nuclei}
\tilde K_{\alpha',\alpha''} =
\begin{dcases}
\frac{\lip \alpha'', \alpha'\rip}{\dim \Upsilon'_{p'q'j}}
 K^{\Upsilon'}_{p'q'j} &\text{when $\Upsilon_{pqj} = \Upsilon'_{p'q'j}$,} \\
 0 &\text{otherwise.}
\end{dcases}
\end{equation}
On the other hand, starting from the definition of $\tilde P_{\alpha',\alpha''}$ it  is easy to check that
\[
K_{\alpha',\alpha''}(z,w)= \int_{\group{U}(n)}\lip \dummy  , \pi(g) \alpha'(w)\rip \, \pi (g) \alpha''(z) \, dg\,.
\]

If now $z_m \alpha$ is decomposed as
\[
z_m \alpha=\sum_{\Upsilon' \in \TT_j} \sum_{p',q'} \alpha^{\Upsilon'}_{p'q'm},
\]
where $\alpha^{\Upsilon'}_{p'q'm} = P^{\Upsilon'}_{p'q'j}(z_m \alpha) \in \Upsilon'_{p'q'j}$, then we observe that
\[\begin{aligned}
\lip z,w\rip K_{\alpha,\alpha} (z,w) 
&= \int_{\group{U}(n)} \lip z,w\rip \, \lip \dummy  , \pi(g) \alpha(w)\rip \, \pi (g) \alpha(z) \, dg\\
&=\int_{\group{U}(n)} \lip \pi(g) z,\pi (g) w\rip \, \lip \dummy  , \pi(g) \alpha(w)\rip \, \pi (g) \alpha(z) dg\\
&=\sum_{m=1}^{n} \int_{\group{U}(n)} \pi(g) z_m \, \overline{\pi (g) w_m}  \, \lip \dummy  , \pi(g) \alpha(w)\rip \, \pi (g) \alpha(z) \, dg\\
&=\sum_{m=1}^{n} \int_{\group{U}(n)} \lip \dummy  , \pi(g) (w_m \alpha)(w)\rip \, \pi(g)( z_m \alpha)(z) \, dg\\
&=\sum_{m=1}^{n} \sum_{\substack{(p',q',\Upsilon') \in I_j \times Y_j \\ (p'',q'',\Upsilon'') \in I_j \times Y_j}} \int_{\group{U}(n)} \biglip \dummy  ,  \pi(g)  \alpha^{\Upsilon'}_{p'q'm}(w) \bigrip \, \pi(g)  \alpha^{\Upsilon''}_{p''q''m}(z) \,dg \\
&= \sum_{m=1}^{n} \sum_{\substack{(p',q',\Upsilon') \in I_j \times Y_j\\ (p'',q'',\Upsilon'') \in I_j \times Y_j}} K_{\alpha_{p'q'm}^{\Upsilon'},\alpha_{p''q''}^{\Upsilon''}}(z,w).
 \end{aligned}\]
By \eqref{eq:relazione-nuclei} the summands in the above expression vanish unless $\Upsilon'_{p'q'j} = \Upsilon''_{p''q''j}$ and we deduce that
\[
\lip z,w\rip
\frac{\lnorm\alpha\rnorm^2}{\dim \Upsilon_{pqj}} \, K_{pqj}^\Upsilon(z,w) 
= \sum_{m=1}^{n} \sum_{(p',q',\Upsilon') \in I_j \times Y_j} \frac{\lnorm\alpha^{\Upsilon'}_{p'q'm}\rnorm^2}{\dim \Upsilon'_{p'q'j}} \, K_{p'q'j}^{\Upsilon'}(z,w),
\]
whence the desired formula for $\delta^{\Upsilon\Upsilon'}_{pp'qq'j}$ follows.
The formula for $\Bar{\delta}^{\Upsilon\Upsilon'}_{pp'qq'j}$ may be proved analogously.
\end{proof}

\begin{corollary}\label{cor:TheoremProjectionv}
Let $0 \leq j \leq n-1$.
With the notation and conventions of Proposition~\ref{prp:TheoremProjectionv},
\[
\sum_{(p,q,\Upsilon) \in I_j \times Y_j} \frac{\dim \Upsilon'_{p'q'j}}{\dim \Upsilon_{pqj}} \, \Bar{\delta}_{pp'qq'j}^{\Upsilon\Upsilon'}
= \sum_{(p,q,\Upsilon) \in I_j \times Y_j}  \frac{\dim \Upsilon'_{p'q'j}}{\dim \Upsilon_{pqj}} \, \delta_{pp'qq'j}^{\Upsilon\Upsilon'} = 1.
\]
\end{corollary}
\begin{proof}
It is sufficient to observe that
\[
\sum_{m=1}^{n} \lnorm\Bar{z}_m \alpha\rnorm^2 = \sum_{m=1}^{n} \lnorm z_m \alpha\rnorm^2 = \sum_{m=1}^{n} \int_{\Sphere} \labs z_m\rabs^2 \, \labs \alpha(z)\rabs^2 \,d\sigma(z) = \lnorm\alpha\rnorm^2
\]
for all $\alpha \in \Upsilon_{pqj}$.
\end{proof}

\begin{corollary}\label{cor:gammarelation}
Let $0 \leq j \leq n-2$.
If $(p,q),(p',q') \in I_j \cap I_{j+1}$, then
\[
\delta^{\Psi\Psi}_{pp'qq'(j+1)} = \frac{\lambda_{p'q'j}^2}{\lambda_{pqj}^2} \, \delta^{\Phi\Phi}_{pp'qq'j} \qquad\text{and}\qquad 
\Bar{\delta}^{\Phi\Phi}_{pp'qq'j} = \frac{\lambda_{p'q'j}^2}{\lambda_{pqj}^2} \, \Bar{\delta}^{\Psi\Psi}_{pp'qq'(j+1)}.
\]
\end{corollary}
\begin{proof}
Note that $\dbarb z_m = 0$ and consequently $\dbarb (z_m \alpha ) = z_m \, \dbarb \alpha$ for all forms $\alpha$.
We deduce that $\dbarb$ commutes with multiplication by $z_m$ and, by considering the formal adjoints, $\dbarba$ commutes with multiplication by $\Bar{z}_m$.

If $\alpha \in \Phi_{pqj}$, then $\dbarb \alpha \in \Psi_{pq(j+1)}$ and
\[
\lnorm \dbarb \alpha \rnorm^2 = \lambda_{pqj}^2 \lnorm \alpha \rnorm^2,
\]
by Proposition \ref{prp:eigenvalues}. 
On the other hand, $\dbarb : \Forms{j} \to \Forms{j+1}$ is $\group{U}(n)$-equivariant, hence
\begin{align*}
\lnorm P^\Psi_{p'q'(j+1)} (z_m \dbarb \alpha) \rnorm^2 
&= \lnorm P^\Psi_{p'q'(j+1)} (\dbarb (z_m \alpha)) \rnorm^2 \\
&= \lnorm \dbarb P^\Phi_{p'q'j} (z_m \alpha) \rnorm^2 
   = \lambda_{p'q'j}^2 \lnorm P^\Phi_{p'q'j} (z_m \alpha) \rnorm^2,
\end{align*}
again by Proposition \ref{prp:eigenvalues}.
Consequently, by Proposition~\ref{prp:TheoremProjectionv},
\begin{align*}
\delta_{pp'qq'(j+1)}^{\Psi\Psi}
&= \frac{\dim \Psi_{pq(j+1)}}{\dim \Psi_{p'q'(j+1)}} \sum_{m=1}^{n} \frac{\lnorm P_{p'q'j}^{\Psi} (z_m \dbarb \alpha) \rnorm^2}{\lnorm\dbarb \alpha\rnorm^2}
\\
&= \frac{\dim \Phi_{pqj}}{\dim \Phi_{p'q'j}}  \frac{\lambda_{p'q'j}^2}{\lambda_{pqj}^2} \sum_{m=1}^{n} \frac{\lnorm P_{p'q'j}^{\Psi} (z_m \dbarb \alpha) \rnorm^2}{\lnorm\dbarb \alpha\rnorm^2} =  \frac{\lambda_{p'q'j}^2}{\lambda_{pqj}^2} \delta_{pp'qq'j}^{\Phi\Phi}.
\end{align*}
The proof of the other formula exploits the properties of $\dbarba$ and is analogous.
\end{proof}

According to Proposition~\ref{prp:TheoremProjectionv}, the coefficients in the decompositions \eqref{eq:deltabar} and \eqref{eq:delta} will be determined once we compute the dimensions of the spaces $\Phi_{pqj}$ and $\Psi_{pqj}$ and we know how the product of an element of $\Phi_{pqj}$ or $\Psi_{pqj}$ with $z_m$ or $\Bar{z}_m$ decomposes as a sum of elements of $\Phi_{p'q'j}$ and $\Psi_{p'q'j}$.
The first problem is easily solved by an application of Weyl's dimension formula (see, for example, \cite[\S 7.1.4, ex.\ 8]{GW2}).

\begin{lemma}\label{lem:dimensions}
Let $0\le j\le n-2$.
Then $\dim \Phi_{pqj}$ is given by
\[
\frac{p+1}{p+n-1-j} \cdot \frac{q}{q+j} \cdot \frac{p+q+n-1}{n-1} \, \binom{n-2}{j} \binom{p+n-1}{n-2} \binom{q+n-2}{n-2} \,.
\]
for all $(p,q)\in I_j$.
The same expression also gives $\dim \Psi_{pq(j+1)}$ for all $(p,q)\in I_{j+1}$.
\end{lemma}

\begin{remark}
The fractions $(p+1)/(p+n-1-j)$ and $q/(q+j)$ are interpreted as $1$ when they are of the form $0/0$.
\end{remark}

For the second problem, instead, we need a more explicit description of the spaces $\Phi_{pqj}$ and $\Psi_{pqj}$. 
Following \cite{Fo}, the representation $\rho(q,\underline{1}_k,\underline{0}_{n-2-k},-p)$ may be identified with a subrepresentation of the canonical representation of $\group{U}(n)$ on
\begin{equation*}
C_{pqk}=\Bigl(
\bigotimes{}^p\, {\C^{n}}^*\Bigr)\otimes
\Bigl( \Lambda^{k+1}\C^n\Bigr)\otimes
\Bigl( \bigotimes{ }^{q-1}\,\C^n\Bigr)\,
\end{equation*}
(here  $0\le k\le n-2$, $ p \geq 0$ and $q \geq 1$).
An  orthonormal basis  for $C_{pqk}$ is given by
\begin{align*}
&\Biglset e^*_{a_1}\otimes \ldots\otimes e^*_{a_p}\otimes \bigl( e_{c_1}\wedge\ldots \wedge e_{c_{k+1}}\bigr) \otimes  e_{b_1}\otimes \ldots \otimes  e_{b_{q-1}}\,: \\
&\qquad\qquad\qquad\qquad
 1\leq a_i\leq n,  1\leq c_1\leq \ldots \leq c_{k+1}\leq n\,,\, 1\leq b_i \leq n \Bigrset\,,
\end{align*}
where $\lset e_i :\, 1\leq i \leq n \rset$ and $\lset e_i^* :\, 1\leq i \leq n \rset$ denote the canonical basis for $\C^n$ and the corresponding dual basis of ${\C^n}^*$.
The space $V_{pqk}$ on which $\rho(q,\underline{1}_k,\underline{0}_{n-2-k},-p)$ acts, that is, $V(q,\underline{1}_k,\underline{0}_{n-2-k},-p)$ in the notation of \cite{Fo}, is the smallest $\group{U}(n)$-invariant subspace of $C_{pqk}$ containing the ``primitive vector''
\begin{equation*}
v_{pqk} 
= \Bigl(\underbrace{ e^*_{n}\otimes\ldots\otimes e^*_n}_{\text{$p$ times}}\Bigr) \otimes \bigl( e_{1}\wedge\ldots \wedge e_{k+1}\bigr) \otimes  \bigl(\underbrace{ e_{1}\otimes\ldots \otimes\ e_{1}}_{\text{$q-1$ times}}\bigr)
\end{equation*}
($P(q,\underline{1}_k,\underline{0}_{n-2-k},-p)$ in the notation of \cite{Fo}).

Similarly one may define $V_{p00}$ (for $p \geq 0)$ and $V_{(-1)q(n-2)}$ (for $q \geq 1$) as the smallest $\group{U}(n)$-invariant subspaces of
\begin{equation*}
C_{p00}
= \bigotimes{ }^p\, {\C^{n}}^*\, 
\qquad\text{and}\qquad
C_{(-1)q(n-2)}
= \Bigl( \Lambda^{n}\C^n\Bigr)\otimes \Bigl( \bigotimes{ }^{q-1}\,\C^n\Bigr)\,
\end{equation*}
containing the primitive vectors
\[
v_{p00} = \underbrace{ e^*_{n}\otimes\ldots\otimes e^*_n}_{\text{$p$ times}}
\quad\text{and}\quad
v_{(-1)q(n-2)} = \bigl(e_1 \wedge \dots \wedge e_n\bigr) \otimes \bigl(\underbrace{  e_{1}\otimes\ldots \otimes e_{1}}_{\text{$q-1$ times}}\bigr)
\]
respectively, and identify $\rho(\underline{0}_{n-1},-p)$ and $\rho(q,\underline{1}_{n-1})$ with the corresponding restrictions of the canonical representations of $\group{U}(n)$.
The sets
\begin{gather*}
\lset e^*_{a_1}\otimes \ldots\otimes e^*_{a_p} \,: \,  1\leq a_i\leq n  \rset
\end{gather*}
and
\begin{gather*}
\lset \bigl( e_{1}\wedge\ldots \wedge e_{n}\bigr) \otimes  e_{b_1}\otimes \ldots \otimes  e_{b_{q-1}}\,: \, 1\leq b_i \leq n \rset
\end{gather*}
are bases for $C_{p00}$ and $C_{(-1) q (n-2)}$ respectively.

Folland identified the spaces $\Phi_{pqj}$ and $\Psi_{pqj}$ with subspaces and quotient spaces of the spaces $C_{pqk}$. 
The index $k$ is equal to $j$ in the case of $\Phi$ and to $j-1$ in the case of $\Psi$. 
The statements of the identifications are a little nicer with a bit more notation: we set $C^\Phi_{pqj} = C_{pqj} $  and $C^\Psi_{pqj} = C_{pq(j-1)}$, and let $v^\Phi_{pqj} = v_{pqj}$ and $v^\Psi_{pqj} = v_{pq(j-1)}$.

Now, when $p\ge 0$ and $q\ge 1$,  we define the map $F^{\Phi}_{pqj}\,:\, C^\Phi_{pqj}\to  \Forms{j}$ by
\begin{align*}
&F^{\Phi}_{pqj}\bigl( e^*_{a_1}\otimes \ldots\otimes e^*_{a_p}\otimes \bigl( e_{c_{1}}\wedge\ldots \wedge e_{c_{j+1}}\bigr) \otimes  e_{b_1}\otimes \ldots \otimes  e_{b_{q-1}}  \bigr)\\
&\qquad =
z_{a_1}\ldots z_{a_p} \Bar{z}_{b_1} \ldots \Bar{z}_{b_{q-1}} \sum_{i=1}^{j+1}(-1)^{i-1} \Bar{z}_{c_{i}} {\zeta}_{c_1}\wedge \ldots    \wedge{\zeta}_{c_{i-1}}  \wedge{\zeta}_{c_{i+1}} \wedge \ldots \wedge {\zeta}_{c_{j+1}}\,,
\end{align*}
(when $j \leq n-2$) and the map  $F^{\Psi}_{pqj}\,:\, C^\Psi_{pqj}\to \Forms{j}$ by
\begin{align*}
&F^{\Psi}_{pqj}\bigl(
e^*_{a_1}\otimes \ldots\otimes e^*_{a_p}\otimes
\bigl( e_{c_1}\wedge\ldots \wedge e_{c_{j}}\bigr) \otimes
 e_{b_1}\otimes \ldots \otimes  e_{b_{q-1}}
 \bigr)\\
&\qquad =
z_{a_1}\ldots z_{a_p}
\Bar{z}_{b_1}
\ldots
\Bar{z}_{b_{q-1}}
{\zeta}_{c_1}\wedge
\ldots \wedge {\zeta}_{c_{j}}
\end{align*}
(when $j \geq 1$), where $\zeta_c = \dbarb \Bar{z}_c$, when $c=1,\dots,n$ (see \cite{Fo}).
Analogously, we set
\[
F^{\Phi}_{p00}\bigl( e^*_{a_1}\otimes \ldots\otimes e^*_{a_p}\bigr)  = z_{a_1}\ldots z_{a_p}\,
\]
and
\begin{align*}
&F^{\Psi}_{(-1)q(n-1)}
\bigl( \bigl( e_{1}\wedge\ldots \wedge e_{n}\bigr) \otimes e_{b_1}\otimes \ldots \otimes  e_{b_{q-1}}
 \bigr)\\
 &\qquad=
\Bar{z}_{b_1} \ldots \Bar{z}_{b_{q-1}} \sum_{i=1}^{n} (-1)^{i+n} \Bar{z}_{i} {\zeta}_{1}\wedge \ldots  \wedge{\zeta}_{i-1}  \wedge{\zeta}_{i+1}  \wedge \ldots \wedge {\zeta}_{n}\,.
\end{align*}
We define $\hwv^{\Upsilon}_{pqj} =  F^{\Upsilon}_{pqj}(v^\Upsilon_{pqj})$; the $\hwv^{\Upsilon}_{pqj}$ are called ``highest weight forms''.

Using Schur's lemma, Folland identified the subspaces $\Upsilon_{pqj}$ in terms of the maps $F^{\Upsilon}_{pqj}$.

\begin{proposition}[{\cite[Theorem 3, Theorem 5 and following]{Fo}}]
Let $0 \leq j \leq n-1$.
For all $(p,q,\Upsilon) \in I_j \times Y_j$, the map $F^{\Upsilon}_{pqj}$ intertwines the $\group{U}(n)$-actions on $C^\Upsilon_{pqj}$ and $\Upsilon_{pqj}$.
Moreover the restrictions  ${F^{\Upsilon}_{pqj}}\rest_{V^\Upsilon_{pqj}}$ are nonzero, since
\begin{align}
\lnorm\hwv^{\Phi}_{pqj}\rnorm^2 &= \frac{2^{j+1}\pi^n p! \, (q-1)!}{(p+q+n-1)!}(q+j)\,\label{eq:phipqj-norm}
\\
\noalign{\noindent{and}}
\lnorm\hwv^{\Psi}_{pqj}\rnorm^2 &= \frac{2^{j+1}\pi^n p! \, (q-1)!}{(p+q+n-1)!}(p+n-j)\label{eq:psipqj-norm}\,.
\end{align}
In particular,
\[
F^{\Upsilon}_{pqj}(V^\Upsilon_{pqj}) = \Upsilon_{pqj} .
\]
\end{proposition}

By using this description, we may prove that the product of $z_m$ or $\Bar{z}_m$ with an element of $\Phi_{pqj}$ or $\Psi_{pqj}$ has at most three nonzero components with respect to the subspaces $\Phi_{p'q'j}$ and $\Psi_{p'q'j}$ and consequently there are at most three nonzero summands in each of the decompositions \eqref{eq:deltabar}, \eqref{eq:delta}.

\begin{lemma}\label{lem:decomp-phi-psi}
Let $0 \leq j \leq n-1$.
For all $m\in\lset1,\ldots,n\rset$ and $(p,q) \in I_j$,
\begin{align}
\Bar{z}_m \Phi_{pqj} &\subseteq \Phi_{p(q+1)j}\oplus \Phi_{(p-1)qj}\,,\label{eq:prodottozbarphi}\\
\Bar{z}_m \Psi_{pqj} &\subseteq \Psi_{p(q+1)j}\oplus \Phi_{pqj}\oplus \Psi_{(p-1)qj}\,,\label{eq:prodottozbarpsi}\\
{z}_m \Phi_{pqj} &\subseteq \Phi_{(p+1)qj}\oplus\Psi_{pqj} \oplus\Phi_{p(q-1)j}\,,\label{eq:prodottozphi}\\
{z}_m \Psi_{pqj} &\subseteq  \Psi_{(p+1)qj}\oplus  \Psi_{p(q-1)j}\,.\label{eq:prodottozpsi}
\end{align}
\end{lemma}
\begin{proof}
To prove \eqref{eq:prodottozbarphi}, we observe that $V_{pqj} \otimes \C^n \subseteq C_{pqj} \otimes \C^n = C_{p(q+1)j}$ and
\begin{equation}\label{eq:tensorproduct}
F^{\Phi}_{p(q+1)j}(v \otimes e_m) = \Bar{z}_m F^{\Phi}_{pqj}(v)
\end{equation}
for all $v \in C_{pqj}$.
In particular
\[
\Bar{z}_m \Phi_{pqj} = \Bar{z}_m F^{\Phi}_{pqj}(V_{pqj}) \subseteq F^{\Phi}_{p(q+1)j}(V_{pqj} \otimes \C^n).
\]
Note that $\rho(1,\underline{0}_{n-1})$ is the canonical representation of $\group{U}(n)$ on $\C^n$.
It is well known (see \cite[\S 18.2.10, formula (1)]{VK} or \cite[\S 7.1.4, ex.\ 5]{GW2}) that the tensor-product representation $\rho(q,\underline{1}_j,\underline{0}_{n-2-j},-p)\otimes \rho(1,\underline{0}_{n-1})$ decomposes as a direct sum of inequivalent irreducible representations;  more precisely
\[
\rho(q,\underline{1}_j,\underline{0}_{n-2-j},-p) \otimes \rho(1,\underline{0}_{n-1})
\cong \bigoplus_{s \in \bar S} \rho((q,\underline{1}_j,\underline{0}_{n-2-j},-p) + (\underline{0}_{s-1},1,\underline{0}_{n-s})),
\]
where $\bar S$ is the set of all $s \in \lset1, \dots, n\rset$ such that $(q,\underline{1}_j,\underline{0}_{n-2-j},-p) + (\underline{0}_{s-1},1,\underline{0}_{n-s})$ is a nonincreasing $n$-tuple.
It is easily seen that at most four direct summands are present; accordingly the representation space decomposes as
\begin{equation}\label{eq:sumVpqj}
V_{pqj} \otimes \C^n = \bigoplus_{l=1}^4 V_{pqj}^{(l)},
\end{equation}
where
\begin{align*}
V_{pqj}^{(1)} \text{ corresponds to } &\rho(q+1,\underline{1}_j,0_{n-2-j},-p),\\
V_{pqj}^{(2)} \text{ corresponds to } &\rho(q,2,\underline{1}_{j-1},0_{n-2-j},-p),\\
V_{pqj}^{(3)} \text{ corresponds to } &\rho(q,\underline{1}_{j+1},0_{n-3-j},-p),\\
V_{pqj}^{(4)} \text{ corresponds to } &\rho(q,\underline{1}_j,0_{n-2-j},-p+1),
\end{align*}
with the convention that $V_{pqj}^{(l)} = \lset0\rset$ if the corresponding $n$-tuple is not nonincreasing.
Note that $V_{pqj}^{(1)} = V_{p(q+1)j} $.

Since $F^{\Phi}_{p(q+1)j} : C_{p(q+1)j} \to \Forms{j}$ is an intertwining operator, a comparison of the decompositions of $V_{pqj} \otimes \C^n$ and $\Forms{j}$ into irreducible representations and an application of Schur's lemma show that
\begin{align*}
F^{\Phi}_{p(q+1)j}(V_{pqj}^{(1)}) &\subseteq \Phi_{p(q+1)j}, & F^{\Phi}_{p(q+1)j}(V_{pqj}^{(2)}) &= \lset0\rset, \\
F^{\Phi}_{p(q+1)j}(V_{pqj}^{(3)}) &= \lset0\rset, & F^{\Phi}_{p(q+1)j}(V_{pqj}^{(4)}) &\subseteq \Phi_{(p-1)qj}
\end{align*}
and \eqref{eq:prodottozbarphi} follows.

To prove \eqref{eq:prodottozbarpsi}, we first check that
$\Bar{z}_m \Psi_{pqj} \subseteq F^{\Psi}_{p(q+1)j}(V_{pq(j-1)} \otimes  \C^n)$, and then apply the decomposition of $V_{pq(j-1)} \otimes \C^n$ given by \eqref{eq:sumVpqj}.
Since $F^{\Psi}_{p(q+1)j}$ is an intertwining operator from $: C_{p(q+1)(j-1)}$ to $\Forms{j}$,  Schur's lemma implies that
\begin{align*}
F^{\Psi}_{p(q+1)j}(V_{pq(j-1)}^{(1)}) &\subseteq \Psi_{p(q+1)j}, & 
F^{\Psi}_{p(q+1)j}(V_{pq(j-1)}^{(2)}) &= \lset0\rset, \\
F^{\Psi}_{p(q+1)j}(V_{pq(j-1)}^{(3)}) &\subseteq \Phi_{pqj}, & 
F^{\Psi}_{p(q+1)j}(V_{pq(j-1)}^{(4)}) &\subseteq \Psi_{(p-1)qj}
\end{align*}
and \eqref{eq:prodottozbarpsi} follows.

To prove \eqref{eq:prodottozphi},  we observe that
\[
z_m \Phi_{pqj} = z_m F^{\Phi}_{pqj}(V_{pqj}) \subseteq F^{\Phi}_{(p+1)qj}({\C^n}^* \otimes V_{pqj}),
\]
since ${\C^n}^* \otimes V_{pqj} \subseteq {\C^n}^* \otimes C_{pqj} = C_{(p+1)qj}$ and $z_m F^{\Phi}_{pqj}(v) = F^{\Phi}_{(p+1)qj}(e^*_m \otimes v)$ for all $v \in C_{pqj}$.
Now $\rho(\underline{0}_{n-1},-1)$ is the canonical representation of $\group{U}(n)$ on ${\C^n}^*$ and
\[
\rho(\underline{0}_{n-1},-1) \otimes \rho(q,\underline{1}_j,\underline{0}_{n-2-j},-p)
\cong \bigoplus_{s\in S}  \rho((q,\underline{1}_j,\underline{0}_{n-2-j},-p)-(\underline{0}_{s-1},1,\underline{0}_{n-s})),
\]
where $S$ is the set of all $s \in \lset1,\dots,n\rset$ such that $(q,\underline{1}_j,\underline{0}_{n-2-j},-p) - (\underline{0}_{s-1},1,\underline{0}_{n-s})$ is a nonincreasing $n$-tuple (see \cite[\S 18.2.10, formula (2)]{VK}).
Accordingly
\begin{equation}\label{eq:sumVmpqj}
{\C^n}^* \otimes V_{pqj} = \bigoplus_{l=1}^4 V_{pqj}^{(-l)},
\end{equation}
where
\begin{align*}
V_{pqj}^{(-1)} \text{ corresponds to } &\rho(q,\underline{1}_j,0_{n-2-j},-p-1),\\
V_{pqj}^{(-2)} \text{ corresponds to } &\rho(q,\underline{1}_{j},0_{n-3-j},-1,-p),\\
V_{pqj}^{(-3)} \text{ corresponds to } &\rho(q,\underline{1}_{j-1},0_{n-1-j},-p),\\
V_{pqj}^{(-4)} \text{ corresponds to } &\rho(q-1,\underline{1}_j,0_{n-2-j},-p),
\end{align*}
with the convention that $V_{pqj}^{(-l)} = \lset0\rset$ if the corresponding $n$-tuple is not nonincreasing.
Note that $V_{pqj}^{(-1)} = V_{(p+1)qj}$.
Again by Schur's lemma we conclude that
\begin{align*}
F^{\Phi}_{(p+1)qj}(V_{pqj}^{(-1)}) &\subseteq \Phi_{(p+1)qj}, & F^{\Phi}_{(p+1)qj}(V_{pqj}^{(-2)}) &= \lset0\rset, \\
F^{\Phi}_{(p+1)qj}(V_{pqj}^{(-3)}) &\subseteq \Psi_{pqj}, & F^{\Phi}_{(p+1)qj}(V_{pqj}^{(-4)}) &\subseteq \Phi_{p(q-1)j}
\end{align*}
and \eqref{eq:prodottozphi} follows.

We prove \eqref{eq:prodottozpsi} analogously, by noting that $z_m \Psi_{pqj} \subseteq F^{\Psi}_{(p+1)qj}({\C^n}^* \otimes V_{pq(j-1)})$ and that
\begin{align*}
F^{\Psi}_{(p+1)qj}(V_{pq(j-1)}^{(-1)}) &\subseteq \Psi_{(p+1)qj}, & F^{\Psi}_{(p+1)qj}(V_{pq(j-1)}^{(-2)}) &= \lset0\rset, \\
F^{\Psi}_{(p+1)qj}(V_{pq(j-1)}^{(-3)}) &= \lset0\rset, & F^{\Psi}_{(p+1)qj}(V_{pq(j-1)}^{(-4)}) &\subseteq \Psi_{p(q-1)j}
\end{align*}
by Schur's lemma.
\end{proof}

Proposition \ref{prp:TheoremProjectionv} and Lemma \ref{lem:decomp-phi-psi} show that all the coefficients not explicitly mentioned in Theorem \ref{thm:Theorem-delta} vanish.
In order to compute the remaining coefficients, the following consequence of the symmetry of Clebsch--Gordan coefficients will be useful (compare with \cite[\S 18.2.1]{VK}).

\begin{lemma}\label{lem:KlymikVilenkin}
Let $\mu$ and $\nu$ be irreducible unitary representations of a compact group $G$ on Hilbert spaces $V^\mu$ and $V^\nu$.
Let $H$ be a minimal nontrivial invariant subspace of $V^\mu \otimes V^\nu$ with respect to the representation $\mu \otimes \nu$ of $G$ and let $\xi$ be the subrepresentation of $\mu \otimes \nu$ on $H$.
Suppose that $\xi$ appears with multiplicity $1$ in $\mu \otimes \nu$.
Let $P_H \in \Lin(V^\mu \otimes V^\nu)$ be the orthogonal projection onto $H$ and $\lset e^\nu_\ell\rset_\ell$ be an orthonormal basis of $V^\nu$.
Then
\[
\sum_{\ell=1}^{\dim V^\nu} \lnorm P_H ( v \otimes e^\nu_\ell ) \rnorm^2 = \frac{\dim H}{\dim V^\mu} \lnorm v\rnorm^2
\]
for all $v \in V^\mu$.
\end{lemma}

We give a proof of this lemma in Section 6.
Now we can determine some of the coefficients in the decomposition.

\begin{proposition}\label{prp:fracdimVpqj}
Let $0 \leq j \leq n-1$.
For all $\Upsilon \in Y_j$ and all $(p,q) \in I_j$,
\begin{align}
\Bar{\delta}^{\Upsilon\Upsilon}_{ppq(q+1)j} &=
\frac{\lnorm\hwv^{\Upsilon}_{p(q+1)j}\rnorm^2}{\lnorm\hwv^{\Upsilon}_{pqj}\rnorm^2},
\label{eq:diagdeltabarphi}\\
\delta^{\Upsilon\Upsilon}_{p(p+1)qqj} &=
\frac{\lnorm\hwv^{\Upsilon}_{(p+1)qj}\rnorm^2}{\lnorm\hwv^{\Upsilon}_{pqj}\rnorm^2} .
\label{eq:diagdeltaphi}
\end{align}
\end{proposition}

\begin{proof}
We consider the case where $\Upsilon = \Phi$.
Let $P^V_{{p(q+1)j}} \in \Lin(V_{pqj}\otimes \C^n)$ be the orthogonal projection onto $V_{p(q+1)j}$.
Then, by \eqref{eq:tensorproduct},
\[\begin{aligned}
\sum_{m=1}^{n} \lnorm P_{p(q+1)j}^{\Phi} (\Bar{z}_m \hwv^{\Phi}_{pqj})\rnorm^2 
&=
\sum_{m=1}^{n} \lnorm P_{p(q+1)j}^{\Phi}
(
F^{\Phi}_{p(q+1)j}
(v_{pqj} \otimes  e_m))\rnorm^2\\
&=\sum_{m=1}^{n}
\lnorm
F^{\Phi}_{p(q+1)j}(P^V_{{p(q+1)j}}
(v_{pqj} \otimes  e_m))\rnorm^2\\
&=
\frac{\lnorm
F^{\Phi}_{p(q+1)j}(
v_{p(q+1)j})
\rnorm^2}{
\lnorm v_{p(q+1)j}\rnorm^2}
\sum_{m=1}^{n}
\lnorm
P^V_{{p(q+1)j}}
(
v_{pqj} \otimes  e_m)\rnorm^2\,,
\end{aligned}\]
where we have repeatedly used the fact that $F^{\Phi}_{p(q+1)j}$ is an intertwining operator, and more precisely that $F^{\Phi}_{p(q+1)j}\bigrest_{V_{p(q+1)j}} : V_{p(q+1)j} \to \Phi_{p(q+1)j}$ is a multiple of a unitary operator.

Note that $F^{\Phi}_{p(q+1)j}(v_{p(q+1)j})=\hwv^{\Phi}_{p(q+1)j}$.
Moreover $\rho(q+1,\underline{1}_j,\underline{0}_{n-j-2},-p)$ is contained once in $\rho(q,\underline{1}_j,\underline{0}_{n-j-2},-p) \otimes \rho(1,\underline{0}_{n-1})$ and $\rho(q,\underline{1}_j,\underline{0}_{n-j-2},-p)$ is contained once in $\rho(q+1,\underline{1}_j,\underline{0}_{n-j-2},-p) \otimes \rho(\underline{0}_{n-1},-1)$ (see the proof of Lemma~\ref{lem:decomp-phi-psi}).
Therefore from Lemma~\ref{lem:KlymikVilenkin}, we deduce that
\[
\sum_{m=1}^{n} \lnorm P_{p(q+1)j}^{\Phi} (\Bar{z}_m \hwv^{\Phi}_{pqj})\rnorm^2 = \lnorm\hwv^{\Phi}_{p(q+1)j}\rnorm^2 \frac{\lnorm v_{pqj}\rnorm^2}{\lnorm v_{p(q+1)j}\rnorm^2} \cdot\frac{\dim V_{p(q+1)j}}{\dim V_{pqj}}.
\]
Since $\lnorm v_{pqj}\rnorm = \lnorm v_{p(q+1)j}\rnorm$, formula \eqref{eq:diagdeltabarphi} follows from Proposition~\ref{prp:TheoremProjectionv}.
The other formulae are proved analogously.
\end{proof}

In the following lemma, we collect some results, which may be easily deduced from \eqref{eq:phipqj-norm}, \eqref{eq:psipqj-norm}, \eqref{eq:deflambdapqj} and Lemma~\ref{lem:dimensions}.

\begin{lemma}\label{lem:collection}
Let $0 \leq j \leq n-1$.
For all $(p,q) \in I_j$, 
\begin{align*}
\frac{\lnorm\hwv^{\Phi}_{(p+1)qj}\rnorm^2}{\lnorm\hwv^{\Phi}_{pqj}\rnorm^2} &= \frac{p+1}{p+q+n},
\\
\frac{\lnorm\hwv^{\Phi}_{p(q+1)j}\rnorm^2}{\lnorm\hwv^{\Phi}_{pqj}\rnorm^2} &= \frac{q}{q+j} \cdot \frac{q+1+j}{p+q+n},
\end{align*}
when $j \leq n-2$, and, if $j \geq 1$, then
\begin{align*}
\frac{\lnorm\hwv^{\Psi}_{(p+1)qj}\rnorm^2}{\lnorm\hwv^{\Psi}_{pqj}\rnorm^2} &= \frac{p+1}{p+n-j} \cdot \frac{p+1+n-j}{p+q+n},
\\
\frac{\lnorm\hwv^{\Psi}_{p(q+1)j}\rnorm^2}{\lnorm\hwv^{\Psi}_{pqj}\rnorm^2} &= \frac{q}{p+q+n}.
\end{align*}
If $0 \leq j \leq n-2$ and $(p,q) \in I_j \cap I_{j+1}$, then
\begin{align*}
\frac{\lambda_{(p+1)qj}^2}{\lambda_{pqj}^2} &= \frac{p+n-j}{p+n-j-1},\\
\frac{\lambda_{p(q+1)j}^2}{\lambda_{pqj}^2} &= \frac{q+1+j}{q+j}.
\end{align*}
Moreover, if $0 \leq j \leq n-2$, then
\begin{align*}
\frac{\dim \Phi_{(p+1)qj}}{\dim \Phi_{pqj}} &= \frac{p+q+n}{p+q+n-1}\cdot\frac{p+n}{p+n-j}\cdot\frac{p+n-j-1}{p+1} = \frac{\dim \Psi_{(p+1)q(j+1)}}{\dim \Psi_{pq(j+1)}},\\
\frac{\dim \Phi_{p(q+1)j}}{\dim \Phi_{pqj}} &= \frac{p+q+n}{p+q+n-1}\cdot\frac{q+j}{q}\cdot\frac{q+n-1}{q+j+1} = \frac{\dim \Psi_{p(q+1)(j+1)}}{\dim \Psi_{pq(j+1)}};
\end{align*}
the equalities on the left hold if $(p,q) \in I_j$ and those on the right if $(p,q) \in I_{j+1}$. 
Finally, if $1 \leq j \leq n-2$ and $(p,q) \in I_j$, then
\begin{align*}
\frac{\dim \Psi_{pqj}}{\dim \Phi_{pqj}} &=
\frac{j}{n-1-j}
\cdot\frac{p+n-j-1}{p+n-j}
\cdot\frac{q+j}{q+j-1}.
\end{align*}
\end{lemma}

Now we can compute \emph{all} the coefficients.

\begin{proof}[Proof of Theorem \ref{thm:Theorem-delta}]
As already observed, the coefficients which are not explicitly mentioned in the statement of Theorem~\ref{thm:Theorem-delta} vanish by Proposition~\ref{prp:TheoremProjectionv} and Lemma~\ref{lem:decomp-phi-psi}.
The remaining coefficients will be computed in four steps.

\emph{Step 1: we find $\Bar{\delta}^{\Phi\Phi}_{ppq(q+1)j}$, $\delta^{\Phi\Phi}_{p(p+1)qqj}$, $\Bar{\delta}^{\Psi\Psi}_{ppq(q+1)j}$, and $\delta^{\Psi\Psi}_{p(p+1)qqj}$.}
These coefficients may be computed immediately by inserting the expressions from Lemma~\ref{lem:collection} into the corresponding formulae in Proposition~\ref{prp:fracdimVpqj}.

\emph{Step 2: we find $\Bar{\delta}^{\Phi\Phi}_{p(p-1)qqj}$ and $\delta^{\Psi\Psi}_{ppq(q-1)j}$.}
By Corollary~\ref{cor:TheoremProjectionv},
\[
\frac{\dim \Phi_{p(q+1)j}}{\dim \Phi_{pqj}} \Bar{\delta}^{\Phi\Phi}_{ppq(q+1)j} + \frac{\dim \Phi_{(p-1)qj}}{\dim \Phi_{pqj}} \Bar{\delta}^{\Phi\Phi}_{p(p-1)qqj} = 1,
\]
hence
\[
\Bar{\delta}^{\Phi\Phi}_{p(p-1)qqj} 
= \frac{\dim \Phi_{pqj}}{\dim \Phi_{(p-1)qj}} \biggl( 1-\frac{\dim \Phi_{p(q+1)j}}{\dim \Phi_{pqj}} \Bar{\delta}^{\Phi\Phi}_{ppq(q+1)j} \biggr).
\]
Analogously,
\[
\delta^{\Psi\Psi}_{ppq(q-1)j} 
= \frac{\dim \Psi_{pqj}}{\dim \Psi_{p(q-1)j}} \biggl( 1-\frac{\dim \Psi_{(p+1)qj}}{\dim \Psi_{pqj}} \delta^{\Psi\Psi}_{p(p+1)qqj} \biggr).
\]
Combining the expressions from Step 1 and Lemma~\ref{lem:collection} with these formulae, we complete Step 2.

\emph{Step 3: we find $\Bar{\delta}^{\Psi\Psi}_{p(p-1)qqj}$ and $\delta^{\Phi\Phi}_{ppq(q-1)j}$.}
The coefficients $\Bar{\delta}^{\Psi\Psi}_{p(p-1)qq(n-1)}$ and $\delta^{\Phi\Phi}_{ppq(q-1)0}$ may be computed as in Step 2.
For the other coefficients, from Corollary~\ref{cor:gammarelation} we see that
\begin{align*}
\Bar{\delta}^{\Psi\Psi}_{p(p-1)qqj} &= \frac{\lambda_{pq(j-1)}^2}{\lambda_{(p-1)q(j-1)}^2} \Bar{\delta}^{\Phi\Phi}_{p(p-1)qq(j-1)},\\
\delta^{\Phi\Phi}_{ppq(q-1)j} &= \frac{\lambda_{pqj}^2}{\lambda_{p(q-1)j}^2} \delta^{\Psi\Psi}_{ppq(q-1)(j+1)},
\end{align*}
and the step follows from these formulae and the expressions from Step 2 and Lemma~\ref{lem:collection}.

\emph{Step 4: we find $\Bar{\delta}^{\Psi\Phi}_{ppqqj}$ and $\delta^{\Phi\Psi}_{ppqqj}$.}
Corollary~\ref{cor:TheoremProjectionv} implies that
\begin{align*}
\Bar{\delta}^{\Psi\Phi}_{ppqqj} 
&= \frac{\dim \Psi_{pqj}}{\dim \Phi_{pqj}} \biggl(1-\frac{\dim \Psi_{p(q+1)j}}{\dim \Psi_{pqj}} \Bar{\delta}^{\Psi\Psi}_{ppq(q+1)j} - \frac{\dim \Psi_{(p-1)qj}}{\dim \Psi_{pqj}} \Bar{\delta}^{\Psi\Psi}_{p(p-1)qqj} \biggr),\\
\delta^{\Phi\Psi}_{ppqqj} 
&= \frac{\dim \Phi_{pqj}}{\dim \Psi_{pqj}} \biggl(1-\frac{\dim \Phi_{(p+1)qj}}{\dim \Phi_{pqj}} \delta^{\Phi\Phi}_{p(p+1)qqj} - \frac{\dim \Phi_{p(q-1)j}}{\dim \Phi_{pqj}} \delta^{\Phi\Phi}_{ppq(q-1)j} \biggr),
\end{align*}
and the results follow by combining these formulae with the expressions from Steps 1 and 3 and Lemma~\ref{lem:collection}.
\end{proof}

\begin{corollary}\label{cor:Corollary-epsilon}
Let $0 \leq j \leq n-1$. 
For all $(p,q,\Upsilon) \in I_j \times Y_j$,
\[
\labs \lip z,w \rip\rabs^2 K^\Upsilon_{pqj}(z,w) = \sum_{(p',q',\Upsilon') \in I_j \times Y_j} \varepsilon^{\Upsilon\Upsilon'}_{pp'qq'j} K^{\Upsilon'}_{p'q'j}(z,w), \\
\]
where in particular
\begin{gather*}
\varepsilon^{\Phi\Phi}_{ppqqj}
= \Bar{\delta}^{\Phi\Phi}_{ppq(q+1)j} \delta^{\Phi\Phi}_{pp(q+1)qj} + \Bar{\delta}^{\Phi\Phi}_{p(p-1)qqj} \delta^{\Phi\Phi}_{(p-1)pqqj},
\\
\varepsilon^{\Psi\Psi}_{ppqqj}
= \delta^{\Psi\Psi}_{p(p+1)qqj} \Bar{\delta}^{\Psi\Psi}_{(p+1)pqqj} + \delta^{\Psi\Psi}_{ppq(q-1)j} \Bar{\delta}^{\Psi\Psi}_{pp(q-1)qj}
\end{gather*}
\end{corollary}

\begin{lemma}\label{lem:conti}
Let $0 \leq j \leq n-1$.
For all $(p,q,\Upsilon) \in I_j \times Y_j$, 
\[
0
<\varepsilon^{\Upsilon\Upsilon}_{ppqqj}<1\,,\qquad 1-\varepsilon^{\Upsilon\Upsilon}_{ppqqj} 
\lesssim \frac{(2+p)(2+q)}{(2+p+q)^2}.
\]
\end{lemma}

\begin{proof}
We consider the case where $\Upsilon=\Phi$ and $1 \leq j \leq n-2$; the other cases may be checked similarly.
By Theorem~\ref{thm:Theorem-delta} and Corollary~\ref{cor:Corollary-epsilon}, $\varepsilon^{\Phi\Phi}_{ppqqj}$ is equal to
\[\begin{gathered}
\frac{q+1+j}{p+q+n} \cdot \frac{q}{q+j} \cdot \frac{q+n-1}{p+q+n-1}
+ \frac{p+n-1}{p+q+n-2} \cdot \frac{p+n-2-j}{p+n-1-j} \cdot \frac{p}{p+q+n-1}\\
=  \frac{q+1+j}{p+q+n} \cdot \frac{q}{q+j} \cdot \theta
+ \frac{p+n-1}{p+q+n-2} \cdot\frac{p+n-2-j}{p+n-1-j} \cdot (1 - \theta) \,
\end{gathered}\]
say, where $0 \leq \theta \leq 1$.
Here $q \neq 0$ since $j \geq 1$.
If $p+n-2-j \neq 0$, then $\varepsilon^{\Phi\Phi}_{ppqqj} $ is a convex combination of two numbers, each of which lies in the interval $(0,1)$, hence  $0<\varepsilon^{\Phi\Phi}_{ppqqj}<1$.
When $p+n-2-j = 0$, one of the summands is $0$ and the other lies in $(0,1)$.

Next  we write
\[\begin{aligned}
1-\varepsilon^{\Phi\Phi}_{ppqqj}
&= M + R_j,
\end{aligned}\]
where
\begin{align*}
M   &= 1- \frac{q}{p+q+n} \cdot\frac{q+n-1}{p+q+n-1} - \frac{p+n-1}{p+q+n-2} \cdot\frac{p}{p+q+n-1},\\
R_j &= - \frac{q}{p+q+n} \cdot \frac{1}{q+j} \cdot \frac{q+n-1}{p+q+n-1} + \frac{p+n-1}{p+q+n-2} \cdot \frac{1}{p+n-1-j} \cdot \frac{p}{p+q+n-1}.
\end{align*}

It is easy to see that $\labs R_j\rabs \lesssim (2+p+q)^{-2} (2+p) (2+q)$ when $0\le j\le n-2$.
Further,
\[\begin{aligned}
M 
&= \biggl( 1- \frac{q}{p+q+n} \biggr) \frac{q+n-1}{p+q+n-1} + \biggl(1 - \frac{p+n-1}{p+q+n-2}\biggr) \frac{p}{p+q+n-1} \\
&= \frac{p+n}{p+q+n} \cdot \frac{q+n-1}{p+q+n-1} + \frac{q-1}{p+q+n-2} \cdot \frac{p}{p+q+n-1}
\end{aligned}\]
and the estimate $M \lesssim (2+p+q)^{-2} (2+p) (2+q)$ follows immediately.

The bounds for $\varepsilon^{\Psi\Psi}_{ppqqj}$  may be established similarly, or alternatively by using the symmetry $\varepsilon^{\Psi\Psi}_{ppqqj} = \varepsilon^{\Phi\Phi}_{(q-1)(q-1)(p+1)(p+1)(n-1-j)}$ (see Remark \ref{rem:symmetry}).
\end{proof}

\section{Plancherel type weighted $L^2$ estimates}

We now prove a ``weighted  Plancherel estimate''.
We shall use the weight $\weight$ given by
\[
\weight(w,z) = \biglabs 1 - \labs\lip w,z\rip\rabs ^2\bigrabs^{1/2}
\qquad\forall w,z \in \C^n\,.
\]

We say that $K$ is a kernel polynomial if
\begin{equation}\label{eq:def-polynomials}
K=\sum_{(p,q,\Upsilon) \in I_j \times Y_j} c^\Upsilon_{pq} K^\Upsilon_{pqj},
\end{equation}
where only finitely many complex coefficients $c_{pq}^\Upsilon$ are nonzero.

We define the linear operators $M^\theta$ on the space of kernel polynomials for all $\theta \in [0,1]$ by
\[
M^\theta K=\sum_{(p,q,\Upsilon) \in I_j \times Y_j} (1-\varepsilon^{\Upsilon\Upsilon}_{ppqqj})^{\theta/2} \, c^\Upsilon_{pq} K^\Upsilon_{pqj},
\]
where $K$ is given by \eqref{eq:def-polynomials} and $\varepsilon^{\Upsilon\Upsilon}_{ppqqj}$ by Corollary \ref{cor:Corollary-epsilon}.
Note that $M^0$ is the identity operator.
We write $M$ in place of $M^1$.

\begin{proposition}\label{prp:L-4.6}
Let $0 \leq j \leq n-1$.
There is a constant $C$, depending only on $n$ and $j$, such that
\[
\lnorm
\labs
  \weight(\dummy ,w) K(\dummy ,w)
  \rabs_{\HS}
  \rnorm_2
\leq C
\lnorm
\labs
M K(\dummy ,w)
  \rabs_{\HS}\rnorm_2 
  \qquad\forall w \in \Sphere
\]
for all kernel polynomials $K$.
\end{proposition}

\begin{proof}
We write $K =\sum_{\Upsilon \in \TT_j} \sum_{\ell=0}^2 K^\Upsilon_\ell$, where $K^\Upsilon_\ell$ consists only of those terms in \eqref{eq:def-polynomials} corresponding to the given $\Upsilon$ for which $p+q\equiv \ell \pmod 3$, that is,
\[
K^\Upsilon_\ell=\sum_{(p,q) \in I_j^\ell}c^\Upsilon_{pq} K^\Upsilon_{pqj}  ,
\]
where $I_j^\ell$ denotes $\lset (p,q) \in I_j : p+q \equiv \ell \pmod 3 \rset$.
Then $M K= \sum_{\Upsilon \in \TT_j} \sum_{\ell=0}^2 M K^\Upsilon_\ell$, and Proposition~\ref{prp:HS-projection} implies that
\[
\lnorm \labs M K(\dummy ,w) \rabs_{\HS}  \rnorm_2^2
= \sum_{\Upsilon \in \TT_j} \sum_{\ell=0}^2 \lnorm \labs M K_\ell^\Upsilon (\dummy ,w) \rabs_{\HS} \rnorm_2^2.
\]
Since
\[
\lnorm \labs M K^\Upsilon_\ell (\dummy ,w) \rabs_{\HS} \rnorm_2 
\leq \lnorm \labs M K(\dummy ,w) \rabs_{\HS} \rnorm_2\,,
\]
for $\Upsilon \in \TT_j$ and $\ell=0,1,2$, it suffices to prove that
\begin{equation}\label{eq:Eq-4.6}
\lnorm
\labs
  \weight(\dummy ,w) K^\Upsilon_\ell (\dummy ,w)
  \rabs_{\HS}
  \rnorm_2
\leq C
\lnorm
\labs
M  K^\Upsilon_\ell (\dummy ,w)
  \rabs_{\HS}\rnorm_2\,.
\end{equation}

We observe that
\begin{equation}\label{eq:intermedia}\begin{aligned}
&\lnorm
\labs
  \weight(\dummy ,w) K^\Upsilon_\ell (\dummy ,w)
  \rabs_{\HS}
  \rnorm_2^2\\
&=
\int_{\Sphere}  \labs K^\Upsilon_\ell (z,w) \rabs_{\HS}^2 \,d\sigma (z)
-\int_{\Sphere}  \labs \lip z,w\rip\rabs^2  \labs K^\Upsilon_\ell (z,w)   \rabs_{\HS}^2 \,d\sigma (z)
\\
&=
\lnorm
\labs
 K^\Upsilon_\ell (\dummy ,w)  \rabs_{\HS}\rnorm_2^2
-\lnorm
 \labs \lip \dummy ,w\rip K^\Upsilon_\ell (\dummy ,w)
  \rabs_{\HS} \rnorm_2^2
 \,.
\end{aligned}\end{equation}

To prove \eqref{eq:Eq-4.6}, note first from Proposition~\ref{prp:HS-projection} that
\begin{equation*}
\lnorm
\labs
 K^\Upsilon_\ell (\dummy ,w)
  \rabs_{\HS}
\rnorm_2^2
=
\sum_{(p,q) \in I_j^\ell} \labs c^\Upsilon_{pq}\rabs^2 \lnorm \labs K^\Upsilon_{pqj}(\dummy ,w)\rabs_{\HS}
 \rnorm_2^2.
\end{equation*}

A similar expression is needed for $\lnorm \labs \lip \dummy ,w\rip  K^\Upsilon_\ell (\dummy ,w)  \rabs_{\HS} \rnorm_2^2 $.
Observe that
\begin{equation*}
\begin{aligned}
&\lnorm
\labs
\lip \dummy ,w\rip
 K^\Upsilon_\ell (\dummy ,w)
  \rabs_{\HS}
\rnorm_2^2 \\
&\qquad=\int_{\Sphere}
 \labs \lip z, w\rip\rabs^2\bigglip \sum_{(p,q) \in I_j^\ell}c^\Upsilon_{pq}
K^\Upsilon_{pqj}(z,w)
,  \sum_{(p',q')\in I_j^ \ell} c^\Upsilon_{p'q'}
K^\Upsilon_{p'q'j}
 (z,w) \biggrip_{\HS} \, d\sigma(z)\\
&\qquad=\sum_{(p,q) \in I_j^\ell}\sum_{(p',q')\in I_j ^ \ell}
 c^\Upsilon_{pq} c^\Upsilon_{p'q'}
 \int_{\Sphere}
\lip
  \lip z,w\rip
K^\Upsilon_{pqj}(z,w)
,
 {\lip z,w\rip}
K^\Upsilon_{p'q'j}
 (z,w) \rip_{\HS} \,d\sigma(z)\,.\\
\end{aligned}
\end{equation*}

We proved in Theorem \ref{thm:Theorem-delta} that
\begin{align*}
\lip z,w \rip K^\Psi_{pqj}(z,w) &= \delta^{\Psi\Psi}_{p(p+1)qqj} K^\Psi_{(p+1)qj}(z,w) + \delta^{\Psi\Psi}_{ppq(q-1)j} K^\Psi_{p(q-1)j}(z,w), \\
\lip z,w \rip K^\Phi_{pqj}(z,w) &= \delta^{\Phi\Phi}_{p(p+1)qqj} K^\Phi_{(p+1)qj}(z,w) + \delta^{\Phi\Phi}_{ppq(q-1)j} K^\Phi_{p(q-1)j}(z,w) 
\\&\qquad
+ \delta^{\Phi\Psi}_{ppqqj} K^\Psi_{pqj}(z,w).
\end{align*}
Hence, from Proposition~\ref{prp:HS-projection}, if $p+q\equiv p'+q' \pmod 3$, then
\[
\int_{\Sphere} \biglip
\lip z,w\rip
K^\Upsilon_{pqj}(z,w)
,
\lip z,w\rip
K^\Upsilon_{p'q'j}
 (z,w) \bigrip_{\HS} \, d\sigma(z)=0
\]
unless $(p,q)=(p',q')$.
Thus, from Corollary \ref{cor:Corollary-epsilon},
\[\begin{aligned}
&\lnorm
\labs
\lip \dummy ,w\rip
 K^\Upsilon_\ell (\dummy ,w)
  \rabs_{\HS}
\rnorm_2^2 \\
&\qquad=\sum_{(p,q) \in I_j^\ell}
\labs  c^\Upsilon_{pq}\rabs^2
 \int_{\Sphere}
\labs   \lip z,w\rip\rabs^2
\biglip K^\Upsilon_{pqj}(z,w)
,
K^\Upsilon_{pqj}(z,w) \bigrip_{\HS} \,d\sigma(z) \\
&\qquad=\sum_{(p,q) \in I_j^\ell}
  \sum_{(p',q',\Upsilon') \in I_j \times Y_j}
\labs  c^\Upsilon_{pq}\rabs^2 \varepsilon^{\Upsilon\Upsilon'}_{pp'qq'j}
 \int_{\Sphere}\biglip K^{\Upsilon'}_{p'q'j}(z,w),
K^\Upsilon_{pqj}(z,w) \bigrip_{\HS} \,d\sigma(z) \\
&\qquad=\sum_{(p,q) \in I_j^\ell}
\labs  c^\Upsilon_{pq}\rabs^2
 \varepsilon^{\Upsilon\Upsilon}_{ppqqj}
 \int_{\Sphere}\biglip  K^\Upsilon_{pqj}(z,w) ,K^\Upsilon_{pqj} (z,w)
\bigrip_{\HS} \,d\sigma(z) \\
&\qquad=\sum_{(p,q) \in I_j^\ell}
\labs  c^\Upsilon_{pq}\rabs^2
 \varepsilon^{\Upsilon\Upsilon}_{ppqqj}
\lnorm \labs  K^\Upsilon_{pqj}(\dummy ,w)\rabs_{\HS}
\rnorm_2^2\,.
\end{aligned}\]

Comparing the expressions for $\lnorm \labs  K^\Upsilon_\ell (\dummy ,w)   \rabs_{\HS} \rnorm_2^2$ and $\lnorm \labs \lip \dummy ,w\rip  K^\Upsilon_\ell (\dummy ,w) \rabs_{\HS} \rnorm_2^2$ above, from \eqref{eq:intermedia} we obtain 
\begin{align*}
&\lnorm
\labs
  \weight(\dummy ,w) K^\Upsilon_\ell (\dummy ,w)
  \rabs_{\HS}
  \rnorm_2^2
\\
&\qquad=
\lnorm
\labs
 K^\Upsilon_\ell (\dummy ,w)  \rabs_{\HS}\rnorm_2^2
-\lnorm
 \labs \lip \dummy ,w\rip\rabs
 \labs K^\Upsilon_\ell (\dummy ,w)
  \rabs_{\HS} \rnorm_2^2 \\
&\qquad=
\sum_{(p,q) \in I_j^\ell} \labs c^\Upsilon_{pq}\rabs^2
\bigl(
1-
 \varepsilon^{\Upsilon\Upsilon}_{ppqqj}\bigr)
\lnorm
\labs
 K^\Upsilon_{pqj} (\dummy ,w)  \rabs_{\HS}\rnorm_2^2
\\
&\qquad=\lnorm
\labs
 M K^\Upsilon_\ell (\dummy ,w)  \rabs_{\HS}\rnorm_2^2\,,
\end{align*}
proving \eqref{eq:Eq-4.6}.
\end{proof}

\begin{corollary}\label{cor:C-5.1v}
There is a constant $C$, depending only on $n$ and $j$, such that 
\begin{equation}\label{eq:convex}
\lnorm
\labs
  \weight(\dummy ,w)\rabs^\theta \labs K(\dummy ,w)
  \rabs_{\HS}
  \rnorm_2
\leq C
\lnorm
\labs
M^{\theta}  K(\dummy ,w)
  \rabs_{\HS}\rnorm_2
  \qquad\forall w \in \Sphere
\end{equation}
for all kernel polynomials $K$ and all $\theta \in [0,1]$.
\end{corollary}
\begin{proof}
Let $T$ be the linear operator that maps a sequence $c = (c_{pq}^\Upsilon)_{(p,q,\Upsilon) \in I_j \times Y_j}$ of complex numbers with a finite number of nonzero terms to the associated kernel polynomial $K$, defined as in \eqref{eq:def-polynomials}. 
Then \eqref{eq:convex} is equivalent to the boundedness of the operator $T$ between suitable weighted Lebesgue spaces.

More precisely, define $\ell^2_\theta$ to be  the weighted space of sequences $c$ such that
\[
\lnorm c \rnorm_\theta = \biggl( \sum_{(p,q,\Upsilon) \in I_j \times Y_j} (1-\varepsilon^\Upsilon_{ppqqj})^{\theta} \frac{\dim \Upsilon_{pqj}}{\sigma(\Sphere)} \labs c_{pq}^\Upsilon\rabs^2 \biggr)^{1/2} < \infty.
\]
Moreover, for any fixed $w \in \Sphere$, let $\FE$ be the vector bundle $\Hom(\Lambda^{0,j}_w,\Lambda^{0,j})$ on $\Sphere$, with fibre $\FE_z = \Hom(\Lambda^{0,j}_w,\Lambda^{0,j}_z)$ endowed with the Hilbert--Schmidt inner product for all $z \in \Sphere$, and define the weighted space of sections $\Gamma_\theta = L^2(\FE,\labs \weight(z,w)\rabs^{2\theta} \,d\sigma(z))$. 
Then, as a consequence of Proposition \ref{prp:HS-projection},
\eqref{eq:convex} is equivalent to the boundedness of the operator $T$ from $\ell^2_\theta$ to $\Gamma_\theta$.

On the one hand,
\[
\lnorm
\labs
  \weight(\dummy ,w)\rabs^{0}\labs K(\dummy ,w)
  \rabs_{\HS}
  \rnorm_2^2=
\lnorm \labs
 K(\dummy ,w)
  \rabs_{\HS}
  \rnorm_2^2
=
\lnorm
\labs
M^{0}  K(\dummy ,w)
  \rabs_{\HS}\rnorm_2^2
  \qquad\forall w \in \Sphere
\]
trivially, and on the other, from Proposition~\ref{prp:L-4.6},
\[
\lnorm
\labs
  \weight(\dummy ,w)\rabs \labs K(\dummy ,w)
  \rabs_{\HS}
  \rnorm_2^2
\leq C
\lnorm
\labs
M K(\dummy ,w)
  \rabs_{\HS}\rnorm_2^2
    \qquad\forall w \in \Sphere\,.
\]
Hence $T$ is bounded from $\ell^2_0$ to $\Gamma_0$ and from $\ell^2_1$ to $\Gamma_1$. 
A standard application of the Stein--Weiss theorem on interpolation with change of measure \cite{StW} to the operator $T$ then yields \eqref{eq:convex} for all $\theta \in [0,1]$.
\end{proof}

As a corollary to Proposition~\ref{prp:L-4.6}, we now prove an analogue of the ``Plancherel-type estimate'' in Assumption 2.5 of \cite{CS}.

For all positive integers $i$, set
\[
H^\Upsilon_{i} = \lset (p,q)\in I_j  : (i-1)^2 \leq (\lambda^\Upsilon_{pqj})^2 \leq i^2\rset\,.
\]

\begin{proposition}\label{prp:P-5.2v}
Suppose that $0 \le \theta<1/2$ and $0 \leq j \leq n-1$.
There is a constant $C$, depending only on $n$ and $\theta$, such that, if $N$ is a positive integer and $K$ is a kernel polynomial as in \eqref{eq:def-polynomials} with $c^\Upsilon_{pq}=0$ if $\lambda^\Upsilon_{pqj} \notin (0,N]$,  then
\begin{equation*}
  \lnorm \labs
  \weight(\dummy ,w)\rabs^{\theta} \labs K(\dummy ,w)
  \rabs_{\HS}
  \rnorm_2^2
\leq C\,N^{2n-1-2\theta}
    \sum_{i=2}^{N}
        C_i 
        \qquad \forall w \in \Sphere,
\end{equation*}
where
\[
C_i = \max\lset\labs c^\Upsilon_{pq} \rabs^2 : \Upsilon \in \TT_j, \, (p,q) \in H^\Upsilon_{i}\rset .
\]
\end{proposition}

\begin{proof}
Fix a kernel polynomial $K$ as in \eqref{eq:def-polynomials} with $c^\Upsilon_{pq}=0$ if $\lambda^\Upsilon_{pqj} \notin (0,N]$.
Note that $(\lambda_{pqj}^\Upsilon)^2$ assumes nonnegative even integer values, hence it is at least $1$ if it does not vanish.
Now Proposition~\ref{prp:HS-projection} implies that
\[\begin{aligned}
\lnorm \labs M^{\theta} K(\dummy ,w)\rabs_{\HS} \rnorm_2^2
&=\sum_{(p,q,\Upsilon) \in I_j \times Y_j}  \bigl(1-\varepsilon^{\Upsilon\Upsilon}_{ppqqj}\bigr)^{\theta} \labs c^\Upsilon_{pq}\rabs^2 \lnorm \labs K^\Upsilon_{pqj}(\dummy ,w) \rabs_{\HS} \rnorm_2^2
\\
&=
\sigma(\Sphere)^{-1} \sum_{(p,q,\Upsilon) \in I_j \times Y_j} \bigl(1-\varepsilon^{\Upsilon\Upsilon}_{ppqqj}\bigr)^{\theta} \labs c^\Upsilon_{pq}\rabs^2  \dim \Upsilon_{pqj}\\
&\leq \sigma(\Sphere)^{-1} \sum_{i=2}^{N} C_i \sum_{\Upsilon \in \TT_j}    \sum_{(p,q)\in H^\Upsilon_{i}}  \bigl(1-\varepsilon^{\Upsilon\Upsilon}_{ppqqj}\bigr)^{\theta} \dim \Upsilon_{pqj} .
\end{aligned}\]
In view of Corollary~\ref{cor:C-5.1v},  it  suffices to prove that
\begin{equation}\label{eq:Hi-estimate}
\sum_{\Upsilon \in \TT_j} \sum_{(p,q)\in H^\Upsilon_{i}}
  \bigl(1-\varepsilon^{\Upsilon\Upsilon}_{ppqqj}\bigr)^{\theta}
    \dim \Upsilon_{pqj}
\lesssim
i^{2(n-\theta)-1}
\end{equation}
for all integers $i \geq 2$.

Recall that 
\[
(\lambda^\Phi_{pqj})^2 = 2(q+j)(p+n-j) 
\qquad\text{and}\qquad
(\lambda^\Psi_{pqj})^2 = 2(q+j-1)(p+n-j+1) .
\]
We may write both expressions for the eigenvalues as $2p'q'$, where $p' = p+n-j$ and $q' = q+j$ in the first case, and $p' = p+n-j+1$ and $q' = q+j-1$ in the second.
When $(p,q) \in I_j \cap H^\Upsilon_{i}$, it is clear that both $p' \geq 1$ and $q' \geq 1$ and also $p'q' \simeq i^2$.
Further, for such $(p,q)$, 
\[
\dim \Upsilon_{pqj}\lesssim (p'q')^{n-2} (p'+q')
\]
 by Lemma~\ref{lem:dimensions}, and 
\[
1-\varepsilon^{\Upsilon\Upsilon}_{ppqqj} \lesssim \frac{p'q'}{(p'+q')^2} 
\]
by Lemma \ref{lem:conti}.
Hence we are reduced to proving that
\begin{equation*}
\sum_{\substack{p' \geq 1, q' \geq 1 \\ (i-1)^2 \leq 2p'q' \leq i^2}} \frac{(p'q')^\theta}{(p'+q')^{2\theta}} \cdot (p'q')^{n-2} (p'+q' ) 
\lesssim i^{2(n-\theta)-1},
\end{equation*}
or equivalently, using the fact that $p'q' \simeq i^2$,
\begin{equation}\label{eq:Hi-estimate-v}
\sum_{\substack{p' \geq 1, q' \geq1 \\ (i-1)^2 \leq 2p'q' \leq i^2}}   \left( \frac{p'+q'}{p'q'} \right)^{1-2\theta} \lesssim i
\end{equation}
for all $k \in \lset0,\dots,n-2\rset$.

By symmetry, it suffices to treat the part of this sum where $p' \geq q'$, and so in particular $q' \leq i$.
Since the number of integer points in an interval is within $1$ of the length of the interval,
\[
\frac{2i-1}{2q'} - 1 \leq \card \lset p' \in \Z: (i-1)^2 \leq 2 p'q' \leq i^2\rset \leq \frac{2i-1}{2q'} + 1 
\]
when $1 \leq q' \leq i$, whence
\begin{align*}
\sum_{\substack{ 1 \leq q' \leq p' \\ (i-1)^2 \leq 2p'q' \leq i^2}}   \left( \frac{p'+q'}{p'q'} \right)^{1-2\theta}
&\leq \sum_{ 1 \leq q' \leq i }  \left(\frac{i}{q'} + 1\right) \left( \frac{2}{q'} \right)^{1-2\theta} \\
&\leq 2^{1-2\theta} i \sum_{ 1 \leq q' \leq i } \left( \frac{1}{q'} \right)^{2-2\theta} + 2^{1-2\theta} \sum_{ 1 \leq q' \leq i } 1 \\
&\lesssim i,
\end{align*}
as required.
The implied constant tends to $\infty$ when $\theta$ tends to $1/2$.
Indeed, using both bounds for the cardinality above shows that the number of integer points in the region between the two hyperbolae under consideration is of the order of $i \log i$.
\end{proof}

Finally, we prove our main results. 
From now on we assume that $1 \leq j \leq n-2$ (and so $n \geq 3$).

\begin{proof}[Proof of Theorem \ref{thm:main}]
Note that  $\lambda_{pqj}^\Upsilon \neq 0$ for all $(p,q,\Upsilon) \in I_j \times Y_j$.
Therefore $\boxb$ on $\Forms{j}$ has trivial kernel and the operator $F(\boxb)$ does not depend on the value of $F$ in $0$; consequently we may assume that the multiplier $F$ vanishes at $0$. 
To conclude, it suffices to apply Theorem~\ref{thm:abstractmult} to the operator $\AbsOp = \boxb$ acting on sections of the bundle $\Lambda^{0,j}$ on the metric measure space $(\Sphere,\dist,\sigma)$, with weight $\varpi = \weight^{\theta}$ and ``dimension constant'' $d = 2n-\theta$, for all $\theta \in (0,1)$.
Let us check that the hypotheses of Theorem~\ref{thm:abstractmult} are satisfied.

It follows from the local behaviour of the distance $\dist$ (see \eqref{eqn:distance-comparison}) that 
\[
\sigma(B(z,r)) \simeq \min(1,r^{2n})
\]
for all $z \in \Sphere$ and $r \in \R_+$ (see \cite[Section 5.1]{Rudin}), and Hypothesis (i) follows. 
Since both $\dist$ and $\weight$ are $\group{U}(n)$-invariant, if $z_0 = (1,0,\dots,0)$ and $w = (w_1,w')$, then
\[
\int_{B(z,t)} \weight(z,w)^{-\theta} \,d\sigma(w) 
= \int_{B(z_0,t)} \labs w'\rabs^{-\theta} \,d\sigma(w) \lesssim t^{2n-\theta},
\]
which is Hypothesis (ii).
Hypothesis (iv) follows from the finite propagation speed property of $\boxb$ (see Section \ref{section:complex}).

To show Hypothesis (iii), from \eqref{eq:operatornorm}, Lemma \ref{lem:Cor-4.3v}, Proposition \ref{prp:HS-projection} and \eqref{eq:Hi-estimate} it follows that
\[\begin{aligned}
\lnorm (1+r^2 \boxb)^{-\ell} \rnorm^2_{L^2(\Sphere) \to L^\infty(\Sphere)} 
&\leq \sum_{(p,q,\Upsilon) \in I_j \times Y_j} (1+(r\lambda_{pqj}^\Upsilon)^2)^{-2\ell} \frac{\dim \Upsilon_{pqj}}{\sigma(\Sphere)} \\
&\lesssim \sum_{i=2}^\infty (1+ r^2 i^2)^{-2\ell} \sum_{\Upsilon \in \TT_j} \sum_{(p,q) \in H_{i}^\Upsilon} \dim \Upsilon_{pqj} \\
&\lesssim \sum_{i=2}^\infty (1+r^2 i^2)^{-2\ell} i^{2n-1} \\
&\lesssim \min(1,r^{2n})^{-1} \simeq \sigma(B(z,r))^{-1}
\end{aligned}\]
for all $z \in \Sphere$ and all $r \in \R$, whenever the integer $\ell$ is greater than $n/2$.

Finally, suppose that the Borel function $F : \R \to \C$ is supported in $[0,N]$.
Then, by Lemma~\ref{lem:Cor-4.3v}, the kernel polynomial $K_{F(\sqrt{\boxb})}$ satisfies the assumptions of Proposition~\ref{prp:P-5.2v}; consequently, for all $\theta \in (0,1)$,
\begin{align*}
&\int_{\Sphere} \labs K_{F(\sqrt{\boxb})}(z,w)\rabs_{\HS}^2 \, \labs \weight(z,w)\rabs^{\theta}\,d\sigma(z) \\
&\qquad\lesssim
 N^{2n-1-\theta} \sum_{i=2}^{N} \max\Biglset \labs F({\lambda_{pqj}^\Upsilon}) \rabs^2 : \Upsilon \in \TT_j, \, (p,q) \in H_{i}^\Upsilon \Bigrset
       \\
&\qquad \lesssim
 N^{2n-\theta}  \frac{1}{N} 
    \sum_{i=2}^{N}
        \sup_{\lambda \in [i-1,i]} \labs F(\lambda) \rabs^2 \\
&\qquad \lesssim N^{2n-\theta}\lnorm F(N \dummy ) \rnorm_{N,2}^2,
\end{align*}
establishing Hypothesis (v).
\end{proof}

\begin{proof}[Proof of Theorem \ref{thm:bochnerriesz}]
If $z \in \C$ and $\Re z > (2n-2)/2$, then the function $F : \R \to \C$ given by $F(\lambda) = (1-\lambda^2)_+^z$ belongs to $H^s(\R)$ for some $s > (2n-\theta)/2$ and $\theta \in (0,1)$, and its $H^s(\R)$-norm may be bounded uniformly in $\Im z$. 
Therefore by Theorem~\ref{thm:abstractmult2} applied with $d = 2n-\theta$ it follows  that,  for all $p \in [1,\infty]$, $(1-t\boxb)_+^{z}$ is bounded on $L^p(\Lambda^{0,j})$ uniformly in $\Im z$ and in $t \in \R_+$. 
On the other hand, if $\Re z \geq 0$, then trivially $(1-t\boxb)_+^{z}$ is bounded on $L^2(\Lambda^{0,j})$ uniformly in $t$ and $z$. 
The result for intermediate values of $\Re z$ then follows by analytic interpolation (see, for example, \cite[\S V.4]{StW2}).
\end{proof}

\section{Some representation theory}
To establish Lemma \ref{lem:KlymikVilenkin}, which we used in Section 4, the following enhanced version of Schur's orthogonality relations  will be useful.
To state the result, we recall that, given a unitary representation of a compact group $G$ on a Hilbert space $V$, and vectors $v,w \in V$, the function $\phi^\pi_{v,w}$ given by
\[
\phi^\pi_{v,w}(x) = \lip \pi(x) v, w \rip
\qquad\forall x \in G,
\]
is called a matrix coefficient of $\pi$. 

\begin{lemma}\label{lem:orthorelations}
Let $\pi$ be a unitary representation of a compact group $G$ on a Hilbert space $V$, and let $H \subseteq V$ be a minimal nontrivial invariant subspace of $H$, such that the subrepresentation of $\pi$ on $H$ has multiplicity $1$ in $\pi$. 
If $v,w \in H$ and $v',w' \in V$, then
\[
\lip \phi^\pi_{v,w}, \phi^\pi_{v',w'} \rip = \frac{\lip v, v' \rip \lip w', w \rip}{\dim H},
\]
where the inner product on the left-hand side is the inner product of $L^2(G)$ with respect to the normalized Haar measure.
\end{lemma}
\begin{proof}
Define $E \in \Lin(V)$  by $Eu = \lip u, w \rip w'$, and let $\tilde E$ be the average of $E$ over $G$, as in the proof of Proposition~\ref{prp:TheoremProjectionv}. 
From the definition of $\tilde E$ we obtain immediately that $\tr \tilde E = \tr E = \lip w', w \rip$ and $\tilde E\rest_{H^\perp} = 0$, since $w \in H$ and $H$ is $\pi$-invariant. 
On the other hand, $\tilde E$ is an intertwining operator for $\pi$ and consequently $\tilde E(H) \subseteq H$ by Schur's lemma, given that the subrepresentation of $\pi$ on $H$ has multiplicity $1$ in $\pi$.
Therefore, again by Schur's lemma, $\tilde E$ is a multiple of the orthogonal projection $P_H$ on $H$ and  $ \tilde E = (\dim H) ^{-1}\lip w', w \rip P_H$ since $\tr P_H = \dim H$.
On the other hand
\[\begin{aligned}
\lip \tilde E v, v' \rip = \int_G \lip \pi(x^{-1}) E \pi(x) v, v' \rip \,dx = \int_G \lip \pi(x) v, w \rip \lip w', \pi(x) w \rip \,dx = \lip \phi^\pi_{v,w} , \phi^\pi_{v',w'} \rip
\end{aligned}\]
and the conclusion follows.
\end{proof}

By using these orthogonality relations, we may now prove the representation-theoretic result that we needed in Section 4.
For the reader's convenience, we restate the lemma.
In this argument, the ranges of sequences and summations are evident, and we usually omit these.

\begin{lemma*}
Let $\mu$ and $\nu$ be irreducible unitary representations of a compact group $G$ on Hilbert spaces $V^\mu$ and $V^\nu$.
Let $H$ be a minimal nontrivial invariant subspace of $V^\mu \otimes V^\nu$ with respect to the representation $\mu \otimes \nu$ of $G$ and let $\xi$ be the subrepresentation of $\mu \otimes \nu$ on $H$. 
Suppose that $\xi$ appears with multiplicity $1$ in $\mu \otimes \nu$.
Let $P_H \in \Lin(V^\mu \otimes V^\nu)$ be the orthogonal projection onto $H$ and $\lset e^\nu_\ell\rset_\ell$ be an orthonormal basis of $V^\nu$. 
Then
\[
\sum_{\ell} \lnorm P_H ( v \otimes e^\nu_\ell ) \rnorm^2 = \frac{\dim H}{\dim V^\mu} \lnorm v\rnorm^2
\]
for all $v \in V^\mu$.
\end{lemma*}
\begin{proof}
First, $\mu$ appears with multiplicity $1$ in $\xi \otimes \tilde\nu$, where $\tilde\nu$ is the contragredient representation to $\nu$. 
This is easily seen by writing the multiplicity as an inner product of characters and exploiting the formulae for characters of tensor products and contragredient representations (see, for example, \cite[\S IV.2, p. 243]{Kn}).

Let $\lset e^\xi_m\rset_m$ be an orthonormal basis of $H$. 
Then
\[
\sum_\ell \lnorm P_H ( v \otimes e^\nu_\ell ) \rnorm^2 = \sum_\ell \sum_m \labs \lip e^\xi_m, v \otimes e^\nu_\ell \rip \rabs^2.
\]
On the other hand, by Lemma~\ref{lem:orthorelations},
\[
\labs \lip e^\xi_m, v \otimes e^\nu_\ell \rip \rabs^2 = \dim H \lip \phi^{\mu\otimes\nu}_{e^\xi_m,e^\xi_m},
\phi^{\mu\otimes\nu}_{v \otimes e^\nu_\ell,v \otimes e^\nu_\ell} \rip
=  \dim H \lip \phi^{\xi}_{e^\xi_m,e^\xi_m}, \phi^{\mu}_{v,v} \, \phi^{\nu}_{e^\nu_\ell,e^\nu_\ell} \rip \,.
\]
Note now that $\overline{\phi^\nu_{e^\nu_\ell,e^\nu_\ell}} = \phi^{\tilde\nu}_{e^{\tilde\nu}_\ell, e^{\tilde\nu}_\ell}$, where $\lset e^{\tilde\nu}_\ell\rset_\ell$ is the dual basis to $\lset e^\nu_\ell\rset_\ell$. 
In particular
\[
\labs \lip e^\xi_m, v \otimes e^\nu_\ell \rip \rabs^2 = \dim H \lip \phi^{\xi}_{e^\xi_m,e^\xi_m} \,
\phi^{\tilde\nu}_{e^{\tilde\nu}_\ell,e^{\tilde\nu}_\ell} , \phi^{\mu}_{v,v} \rip = \dim H \lip \phi^{\xi\otimes\tilde\nu}_{e^\xi_m\otimes e^{\tilde\nu}_\ell,e^\xi_m \otimes e^{\tilde\nu}_\ell} , \phi^{\mu}_{v,v} \rip.
\]
Since $\mu$ occurs in $\xi \otimes \tilde\nu$ with multiplicity $1$, there exists a subspace $W$ of $H \otimes (V^\nu)^*$ such that the subrepresentation $\eta$ of $\xi \otimes \tilde\nu$ on $W$ is equivalent to $\mu$, and consequently there exists $w \in W$ such that $\lnorm w\rnorm = \lnorm v\rnorm$ and $\phi^\eta_{w,w} = \phi^\mu_{v,v}$. 
In particular, again by Lemma~\ref{lem:orthorelations},
\[
\labs \lip e^\xi_m, v \otimes e^\nu_\ell \rip \rabs^2 = \frac{\dim H}{\dim W} \labs \lip e^\xi_m\otimes e^{\tilde\nu}_\ell, w\rip\rabs^2.
\]
(For irreducible representations, and vectors in orthonormal bases of the representation spaces, this property is known as the symmetry of Clebsch--Gordan coefficients---see, for instance, \cite[\S18.2.1]{VK}).
Since $\lset e^\xi_m\otimes e^{\tilde\nu}_\ell\rset_{m,\ell}$ is an orthonormal basis of $H \otimes (V^\nu)^*$, by summing the above equality over $m$ and $\ell$ we obtain
\[
\sum_\ell \lnorm P_H ( v \otimes e^\nu_\ell ) \rnorm^2 = \frac{\dim H}{\dim W} \lnorm w\rnorm^2 = \frac{\dim H}{\dim V^\mu} \lnorm v\rnorm^2,
\]
and we are done.
\end{proof}

\end{document}